\documentclass[12pt]{amsart}

\usepackage{amsmath,amssymb,amsthm,amscd,amsfonts}
\usepackage{fullpage}
\usepackage{graphicx}
\usepackage{stmaryrd}
\usepackage{float}
\usepackage{hyperref}

\usepackage{algorithm}							
\usepackage[noend]{algpseudocode}					
\usepackage{caption}
\usepackage[margin=1in]{geometry}
\usepackage{url}
\usepackage{xcolor}

\usepackage{mathtools}

\DeclarePairedDelimiter\floor{\lfloor}{\rfloor}

\usepackage{amssymb,latexsym}
\usepackage{amsmath,amsfonts,amsthm,amscd,amsxtra,enumerate}
\usepackage{mathtools}
\usepackage{caption}
\usepackage[all]{xy}
\usepackage{mathrsfs}

\usepackage[margin=1in]{geometry}
\usepackage{epsfig}
\usepackage{pgf,tikz}
\usepackage{multicol}
\usepackage[utf8]{inputenc}
 \usepackage{subcaption}

\usepackage{enumitem}
\usepackage{algpseudocode}
\usepackage{algorithm}

\usepackage{amssymb,stmaryrd,mathrsfs,dsfont,mathtools,enumitem}
\usepackage{hyperref}
\usepackage[noend]{algpseudocode}					
\usepackage{url}
\usepackage{xcolor}

\usepackage{mathtools}

\usepackage[all]{xy}

\usepackage{epsfig}
\usepackage{pgf,tikz}
\usepackage{multicol}
\usepackage[T1]{fontenc}
\usepackage{enumitem}

\newtheorem{conjecture}{ Conjecture}[section]
\newtheorem{theorem}[conjecture]{ Theorem}
\newtheorem{lemma}[conjecture]{ Lemma}
\newtheorem{corollary}[conjecture]{ Corollary}
\newtheorem{proposition}[conjecture]{ Proposition}
\theoremstyle{definition}
\newtheorem{remark}[conjecture]{ Remark}

\newtheorem{definition}[conjecture]{ Definition}

\newtheorem{example}[conjecture]{ Example}

\providecommand\max{\text{\rm max}}

\allowdisplaybreaks

\begin{document}

\title{Mutually Abelian-Bordered Binary Words}

\author{Anuran Maity$^1$ \and K. V. Krishna$^1$}

\address
	{$^1$Department of Mathematics\\ 
	 Indian Institute of Technology Guwahati, 
	  Guwahati, India}

      \email{anuran.maity@gmail.com; kvk@iitg.ac.in}


%
%


	\begin{abstract}
		A word is said to be bordered if it contains a nonempty proper prefix that is also a suffix. A pair of words
		$(u, v)$ is said to be mutually bordered if there exists a word that is a nonempty proper prefix of $u$ and suffix of $v$, and there exists a word that is a nonempty proper suffix of $u$ and prefix of $v$. Recently, Gabric studied the number of mutually bordered pairs. In this work, we extend the concept of mutually bordered pairs to abelian setting, and determine the number of  mutually abelian-bordered pairs of binary words using lattice paths. We also find the number of unbordered pairs in this context.
	\end{abstract}
\maketitle

	\section{Introduction}
	
	Two words $x$ and $y$ over an alphabet $\Sigma$ are said to be abelian equivalent, denoted by $x \sim_{\textup{abl}} y$,  if $|x|_a = |y|_a$ for all $a \in \Sigma$, where $|x|_a$ denotes the number of occurrences of $a$ in $x$.  Note that if $x \sim_{\textup{abl}} y$ then $|x| = |y|$, i.e., they are of the same length. For example, $baba \sim_{\textup{abl}} abba$.  The abelian equivalence relation forms the foundation for abelian combinatorics on words. 
	A natural goal of abelian combinatorics on words is to extend the key results from the classical non-commutative setting to the abelian context.
	In recent years, substantial research has been conducted in this direction (see, e.g., \cite{aberkane2004number}, \cite{avgustinovich2016weak},  \cite{carpi1998number}, \cite{cassaigne2011avoiding}, \cite{currie2001avoiding},   \cite{currie2012fixed}, \cite{currie2008long},
	\cite{fici2017abelian}, \cite{fici2019abelian}, \cite{peltomaki2020avoiding},
	\cite{richmond2008counting}). 
	In this paper, we contribute to this line of research by addressing a problem related to the abelian setting of mutually bordered words.
	
	A word $w$ is said to be bordered if it has a nonempty proper prefix which is also its suffix; otherwise, $w$ is called unbordered. A word $w$ is said to have an abelian-border if it has a nonempty proper prefix that is abelian equivalent to a suffix; in this case, $w$ is called an abelian-bordered word. If a word is not abelian-bordered, it is called abelian-unbordered. The notion of bordered as well as abelian-bordered words have been extensively studied in the literature (see, e.g., \cite{charlier2016abelian}, \cite{christodoulakis2014abelian}, \cite{ehrenfeucht1979periodicity}, \cite{harju2004border}, \cite{harju2004periodicity},  \cite{holub2017fully}, \cite{nielsen1973note}, \cite{rampersad2013number}). In particular, Nielsen \cite{nielsen1973note} analyzed the number of length $n$ unbordered words.  Inspired by the applications in digital communication, in  \cite{gabric2022mutual},  Gabric extended the results of Nielsen to a generalized notion called mutually (un)bordered pairs of words.
	A pair of words $(u, v)$ is said to be mutually bordered if there exists a word that is a nonempty proper prefix of $u$ and suffix of $v$, and there exists a word that is a nonempty proper suffix of $u$ and prefix of $v$. In this paper, we extend the work of Gabric to the abelian setting. For the concepts and notations that are not defined in this paper, one may refer to \cite{lothaire}. 
	
	\begin{definition}
		We say a pair of words $(x, y)$ is an internal abelian-border of $(u, v)$  
		if $x$ is a nonempty proper suffix of $u$ and  $y$ is a proper prefix of $v$ such that $x \sim_{\textup{abl}} y$. Similarly, we say the  pair $(x, y)$ is an external abelian-border of $(u, v)$ 
		if $x$ is a nonempty proper prefix of $u$ and  $y$ is a proper suffix of $v$ such that $x \sim_{\textup{abl}} y$. A pair of words $(u, v)$ is said to be mutually abelian-bordered if
		$(u, v)$ has both internal abelian-border and external abelian-border. 
	\end{definition}
	  	
	\begin{definition}
		If a pair of words $(u, v)$ has neither an internal abelian-border nor an external abelian-border, then $(u, v)$ is said to be mutually abelian-unbordered pair of words. A pair of words $(u, v)$ is said to be internal abelian-bordered if $(u, v)$ has an internal abelian-border but does not have an external abelian-border.
		Similarly, a pair $(u, v)$ is said to be  external abelian-bordered if $(u, v)$ has an external abelian-border but does not have an internal abelian-border.
	\end{definition}
	
	\begin{example}
		The pair $(a abab, aabb a)$ is mutually abelian-bordered,  as $(a, a)$ and $(abab, aabb)$ are its external abelian- and internal abelian-borders, respectively.
		The pair $(aabb, abbb)$ is internal abelian-bordered, as $(abb, abb)$ is its internal abelian-border and it has no external abelian-border. 
		The pair $(abbb, aabb)$ is external abelian-bordered, as  $(abb, abb)$ is its external abelian-border and it has no internal abelian-border. Further, the pair
		$(aaa, bbb)$ is mutually abelian-unbordered as it has no internal and external abelian-borders.
	\end{example}

	 We write $\textup{sb}(u, v)$ to denote the shortest internal abelian-border of $(u, v)$. Further, if $\textup{sb}(u, v) = (x, y)$, then we use $\textup{lsb}(u, v)$ to denote $|x|$, the length of shortest internal abelian-border of $(u, v)$. Note that $(x, y)$ is an internal abelian-border of $(u, v)$ if and only if $(y, x)$ is an external abelian-border of $(v, u)$. 
	
	\begin{remark}
		\begin{enumerate}
			\item $\textup{sb}(v, u)$ is the shortest external abelian-border of $(u, v)$.
			\item A pair $(u, v)$ does not have an internal abelian-border iff $\textup{sb}(u, v) = (\lambda, \lambda)$.
			\item For $u, v \in \Sigma^n$, if $\textup{lsb}(u, v)=i$ and $\textup{lsb}(v, u)=j$, then $i\leq n-1$, $j \leq n-1$ and $i+j \leq 2n-2$.
			\item If $\textup{sb}(u, v)=(x, y)$, then $(x, y)$ is mutually abelian-unbordered.
			\item If $(x, y)$ is the shortest external abelian-border of $(u, v)$, then $(x, y)$ is mutually abelian-unbordered.
		\end{enumerate}
	\end{remark}
	
	In this paper, we consider $\Sigma = \{a, b\}$ and if a pair of equal length words over $\Sigma^*$ is mutually abelian-bordered, in short, we refer it as an MAB pair. If the pair is mutually abelian-unbordered, then it is referred as an MAU pair. In this work, we compute the number of MAB as well as MAU pairs. 

    \section{Conputing MAB pairs}
	
	The number of MAB pairs $(u, v)$ with $|u| = |v| = n$ is denoted by $\mathcal{M}(n)$. It is evident that $\mathcal{M}(1)=0$.
	We compute $\mathcal{M}(n)$ for $n \geq 2$. For an MAB pair $(u, v)$, we write $i = \textup{lsb}(u, v)$ and $j = \textup{lsb}(v, u)$, and both are always positive. 
	If the borders are disjoint in the words, i.e., $i + j \leq n$, we denote the corresponding number with $\mathcal{M}^{\textup{d}}(n)$. 
	Similarly, if the borders are overlapping in the words, i.e., $i + j > n$, then the corresponding number is denoted by  $\mathcal{M}^{\textup{o}}(n)$, so that $\mathcal{M}(n) = \mathcal{M}^{\textup{d}}(n) + \mathcal{M}^{\textup{o}}(n)$. We compute $\mathcal{M}^{\textup{d}}(n)$ and $\mathcal{M}^{\textup{o}}(n)$ in sections \ref{subsection1} and \ref{subsection2}, respectively. 
	
	\subsection{Computing $\mathcal{M}^{\textup{d}}(n)$}\label{subsection1}
	
	We first discuss the structure of MAB pairs $(u, v)$ such that $i + j \leq n$ in the following lemma. We omit its proof as it is straightforward from the definitions.  
	
	\begin{lemma}\label{less1}
		For an MAB pair $(u, v)$ with $|u|=|v|=n \geq 2$, $i + j \leq n$ iff $u= \alpha x \beta$ and $v= \beta' x' \alpha'$ for some $\alpha, \beta, \beta', \alpha' \in \Sigma^+$ and $x, x' \in \Sigma^*$ with $\beta \sim_{\textup{abl}} \beta'$, $\alpha \sim_{\textup{abl}} \alpha'$, $|\alpha|=j$, $|\beta|=i$, and both $(\beta, \beta')$ and $(\alpha, \alpha')$ are MAU pairs. 
	\end{lemma}
	

	We recall the following result on the number of binary words of length $m$ with the shortest abelian-border of length $k$, denoted by
	$D(m, k)$.
	\begin{theorem}[\cite{blanchet2022dyck}]\label{numthe}
		For $m \ge 2$, $D(m, k) =\frac{1}{k}{2k-2 \choose k-1}  2^{m-2k+1}$.
	\end{theorem}
	
	\begin{theorem}\label{mdn}
		For $n \geq 2$, 
		$\displaystyle \mathcal{M}^{\textup{d}}(n) = \sum_{i=1}^{n-1} \sum_{j=1}^{n-i} D(2n-2i, j) D(2i, i).$
	\end{theorem}
	
	\begin{proof}
		Let $(u, v)$ be an MAB pair with $i + j \leq n$. In view of Lemma \ref{less1}, $uv= \alpha x \beta \beta' x' \alpha'$ and note that $\beta \beta'$ is an abelian-bordered word with the shortest abelian-border of length $|\beta|$. The number of such words in $uv$ is $D(2i, i)$. Also, $\alpha x x' \alpha'$ is an abelian-bordered word with the shortest abelian-border of length $|\alpha|$ in $uv$. Now the number of words of the form  $\alpha x x' \alpha'$ is $D(2n-2i, j)$. Hence, the total number of desired pairs $(u, v)$ is $\displaystyle \sum_{i=1}^{n-1} \sum_{j=1}^{n-i} D(2n-2i, j) D(2i, i)$. 
	\end{proof}
	We have verified Theorem \ref{mdn} for some initial values through a computer program, and the results are given in Table \ref{tab:sample_table}. 
	
	
	\begin{table}[h]
		\centering
		\begin{tabular}{|c|c|c|c|c|c|c|c|c|c|c|c|c|}
			\hline
			$n$ & 1 & 2 & 3 & 4 & 5 & 6 & 7 & 8 & 9 & 10 & 11 & 12 \\
			\hline
			$\mathcal{M}^{\textup{d}}(n)$ & 0 & 4 & 24 & 116 & 520 & 2248 & 9520 & 39796 & 164904 & 679064 & 2783440 & 11368904   \\
			\hline
		\end{tabular}
		\caption{Values of $\mathcal{M}^{\textup{d}}(n)$ for $1 \leq n \leq 12$ }
		\label{tab:sample_table}
	\end{table}
	

	\subsection{ Computing $\mathcal{M}^{\textup{o}}(n)$}\label{subsection2}

	We first present a straightforward property on the structure of MAB pairs $(u, v)$ such that $i + j > n$ in the following lemma.
	

	\begin{lemma}\label{gret}
		For an MAB pair $(u, v)$ with $|u|=|v|=n \geq 3$, $i+j > n$ iff $u= {\alpha} {\gamma} {\beta}$ and $v= {\beta'} {\gamma'} {\alpha'} $ for some ${\alpha}, {\gamma}, {\beta}, {\beta'}, {\gamma'}, {\alpha'}  \in \Sigma^+$  with ${\alpha} {\gamma} \sim_{\textup{abl}} {\gamma'} {\alpha'}$, ${\gamma} {\beta} \sim_{\textup{abl}} {\beta'} {\gamma'}$, $|{\alpha} {\gamma} |=j$, $|{\gamma} {\beta} |=i$, and both $({\gamma} {\beta}, {\beta'} {\gamma'})$ and $({\alpha} {\gamma}, {\gamma'}{\alpha'})$ are MAU pairs.
	\end{lemma}
	
	To calculate $\mathcal{M}^{\textup{o}}(n)$, we use the relationship between binary words and the lattice paths.
	A \textit{lattice path} is a sequence of one unit steps between two points in $\mathds{Z}^2$ such that each step is towards right (east) or up (north). We depict lattice paths in $XY$-plane by considering the lattice points, i.e., points with integer coordinates. Representing each right move by one letter and a move upward by another letter, and vice versa, we can see that there is a bijection between lattice paths and binary words. Fig. \ref{entrypre}  gives an example of lattice paths and corresponding words. We use the following notations throughout the text.

	
	\begin{enumerate}
		\item For a binary word $w$, we denote its corresponding lattice path which starts at a point $A$ by $p_A(w)$. Given a lattice path $p$ starting from a point $A$, the corresponding binary word  is denoted by $w_A(p)$; if the end point $B$ is also specified, then the word is denoted by $w_A^B(p)$. 
		
		\item Let $p$ and $q$ be two lattice paths.  The set of all lattice points common to them is denoted by $p \cap q$. If $q$ starts where $p$ ends, then $pq$ denotes the path which starts at the starting point of $p$ and ends at the ending point of $q$.   
		
		\item If $A$ is point in $XY$-plane, then $X(A)$ and $Y(A)$ denote the $X$- and $Y$-coordinates of $A$, respectively.

		
		
		\item For $A = (X_1, Y_1), B=(X_2, Y_2) \in \mathds{Z}^2$, the number of lattice paths from $A$ to $B$ is denoted by $N_{A}^{B}$. Let $\Delta(X)=X_2 - X_1 $ and $\Delta(Y) = Y_2 - Y_1$. If both $\Delta(X),  \Delta(Y) \ge 0$, then $N_{A}^{B} = { X_2-X_1 + Y_2 - Y_1 \choose X_2 - X_1}$; else,  $N_{A}^{B} =0$. 
		
		\item For distinct lattice points $A, B$ and $C$, let $N_{A, B}^{C}$ denote the number of pairs of lattice paths $(p, q)$ such that $p$ is from $A$ to $C$, $q$ is from $B$ to $C$, and $p \cap q = \{ C\}$. 
		Also,  we write $N_{A}^{B, C}$ to denote the number of pairs of lattice paths $(p', q')$ such that $p'$ is from $A$ to $B$, $q'$ is from $A$ to $C$, and $p' \cap q' = \{ A\}$. 
		
		
		\item Let $A, B$ and $C$ be distinct lattice points such that $X(A)<X(B)$. We write ${}_{A}\mathcal{N}_{B}^{C}$ to denote the set of all triplets of lattice paths $(p, q, p')$ such that $p$ is from $A$ to $C$, $q$ is from $B$ to $C$, $p'$ is from $B$ corresponding to the word $w_{A}^{C}(p)$, $p\cap q = \{C\}$, and $q \cap p' = \{B\}$.   Further, let $|{}_{A}\mathcal{N}_{B}^{C}| = {}_{A}{N}_{B}^{C}$. 
		
		\item  Let $L$ be a straight line intersecting the two distinct lattice paths $p$ and $p'$. We define the \emph{internal points} on $L$ with respect to $p$ and $p'$ as those lattice points on $L$ which lie in between $p$ and $p'$, but not on either $p$ or $p'$. We denote the set of all such points by $\mathtt{IP}_{p,p'}(L)$.
		Also, we write $\mathtt{IP}_{p,p'}(L ~\&~ X = m)$ to denote the set of all internal points $(X_1, Y_1)$ on $L$ such that $X_1 \leq m$.
	\end{enumerate}

	We recall the relation between abelian-border of a word and lattice paths.
	\begin{theorem}[\cite{rampersad2013number}] \label{Lpth1}
		A binary word $w$ has an abelian-border of length $k$ iff for any point $A$ in $\mathds{Z}^2$, the lattice paths $p_A(w)$ and $p_A({w^R})$ meet after $k$ steps.
	\end{theorem}

	\begin{remark}\label{onb}
		We have the following straightforward observations from Theorem \ref{Lpth1}.
		\begin{enumerate}
			\item \label{onb1}
			The length of the shortest abelian-border of a word $w$ is $k$ iff for any lattice point $A$, the paths $p_A(w)$ and $p_A({w^R})$ meet again only after $k$ steps.
			
			\item \label{onb2}
			Let $u, v \in \Sigma^+$ and $A \in \mathds{Z}^2$. The lattice paths $p_A(u)$ and $p_A({v})$  intersect after $k$ steps iff there exist $u' \in \text{Pref}_k(u)$ and $v' \in \text{Pref}_k(v)$ such that $u' \sim_{\textup{abl}} v'$.
		\end{enumerate}
	\end{remark} 
	
	\begin{example}
		Let $w= abab a aabb$. Clearly, $w$ has abelian-borders of lengths $4$ and $5$.
		As shown in Fig. \ref{entrypre}, $p_O(w)$ and $p_O({w^R})$ intersect at $A$ and $B$ (except the end points) after $4$ and $5$ steps, respectively. Note that these paths meet for second time at $A$ after $4$ steps, and length of the shortest abelian-border of $w$ is $4$.

	\end{example}
	
	\begin{figure}[h]
		\centering
		\begin{tikzpicture}
			\draw[step=0.5cm,gray,very thin] (0,0) grid (3,3); 
			\draw[line width=0.5mm]  (0,0) -- (0,0) node[below] {$O$};

			\draw[line width=0.5mm]  (0,0) -- (0.5,0) node[right] {};
			\draw[line width=0.5mm]  (0.5,0) -- (0.5,0.5) node[right] {};
			\draw[line width=0.5mm]  (0.5,0.5) -- (1,0.5) node[right] {};
			\draw[line width=0.5mm]  (1,0.5) -- (1,1) node[right] {};
			\draw[line width=0.5mm]  (1,1) -- (1.5,1) node[right] {};
			\draw[line width=0.5mm]  (1.5,1) -- (2,1) node[right] {};
			\draw[line width=0.5mm]  (2,1) -- (2.5,1) node[right] {};
			\draw[line width=0.5mm]  (2.5,1) -- (2.5,1.5) node[right] {};
			\draw[line width=0.5mm]  (2.5,1.5) -- (2.5,2) node[right] {};
			\draw[red, line width=1mm]  (0,0) -- (0,0.5) node[right] {};
			\draw[red, line width=1mm]  (0,0.5) -- (0,1) node[right] {};
			\draw[red, line width=1mm]  (0,1) -- (0.5,1) node[right] {};
			\draw[red, line width=1mm]  (0.5,1) -- (1,1) node[right] {};
			\draw[red, line width=1mm]  (1,1) -- (1.5,1) node[right] {};
			\draw[red, line width=1mm]  (1.5,1) -- (1.5,1.5) node[right] {};
			\draw[red, line width=1mm]  (1.5,1.5) -- (2,1.5) node[right] {};
			\draw[red, line width=1mm]  (2,1.5) -- (2,2) node[right] {};
			\draw[red, line width=1mm]  (2,2) -- (2.5,2) node[right] {};
			\draw[line width=0.5mm]  (2.5,2) -- (2.5,2) node[right] {$C$};

			\fill[] (1,1) circle (2pt);
			\draw[line width=0.5mm]  (1,1) -- (1,1) node[above] {$A$};
			\fill[] (1.5,1) circle (2pt);
			\draw[line width=0.5mm]  (1.5,1) -- (1.5,1) node[below] {$B$};
			
		\end{tikzpicture}
		\caption{For $w=ababaaabb$, lattice paths $p_O(w)$ (thin) and  $p_O({w^R})$ (thick).}
		\label{entrypre}
	\end{figure}

	\begin{remark}\label{bijec}
		The function $f: {\Sigma}^{n} \rightarrow P_{A}(n)$ defined by $f(w) = p_A({w})$ is a bijection, where $P_{A}(n)$ is the set of all lattice paths from a point $A = (k, l)$ to the line $X+Y= k+l+n$. 
	\end{remark}

	
	\begin{figure}[htbp]
		\centering
			\includegraphics[width=\linewidth]{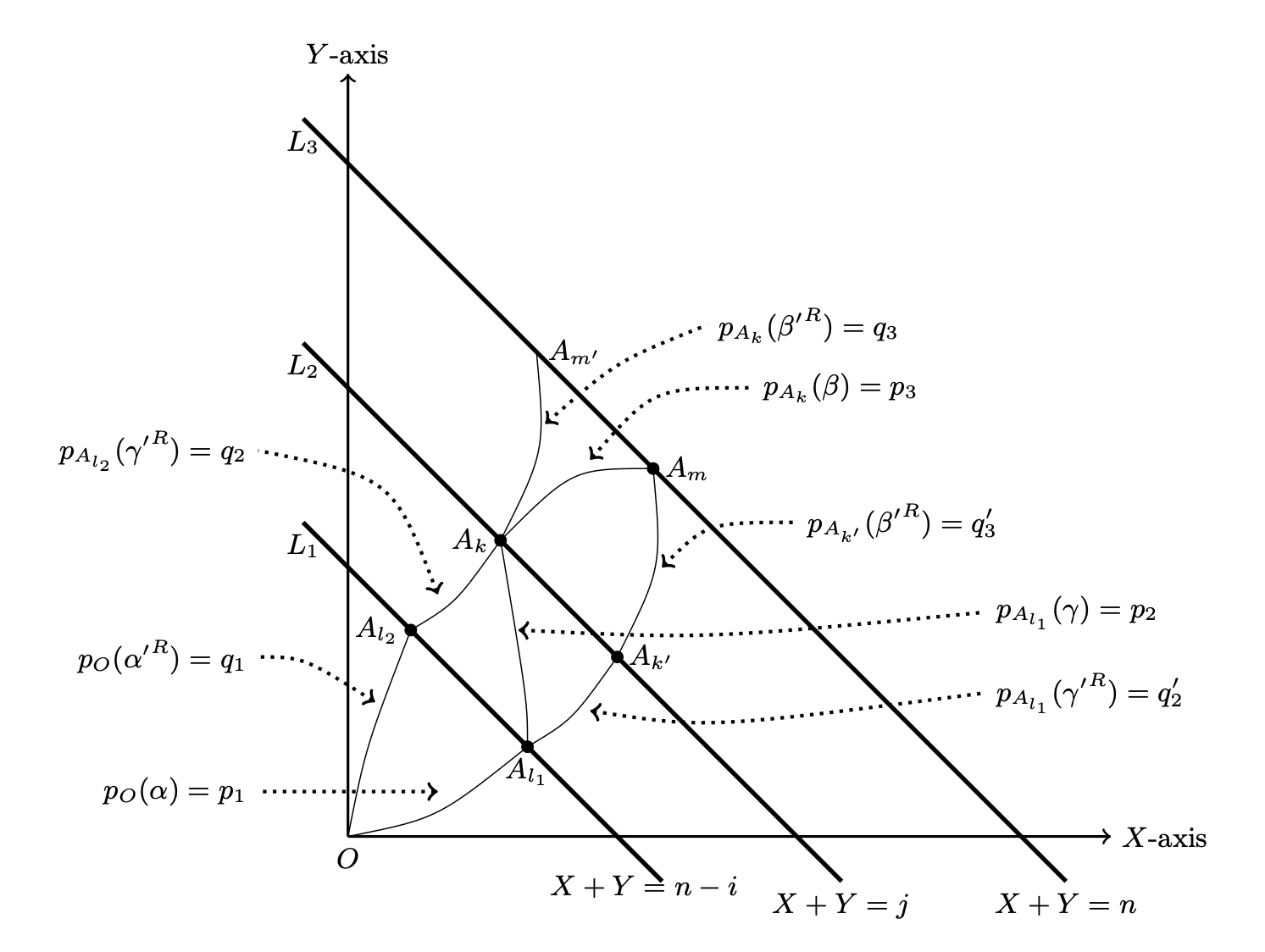}
			\caption{Lattice paths $p_O({\alpha {\gamma} {\beta}})$ and $p_O({{\alpha'}^R {\gamma'}^R {\beta'}^R})$}.
			\label{yytr}
	\end{figure}
	
\subsubsection{Strategy of Computation}
	We now describe the strategy of counting $\mathcal{M}^{\textup{o}}(n)$. Let $(u, v)$ be an arbitrary MAB pair with where $|u|=|v|=n$. In view of Lemma \ref{gret}, we have the following structure of $(u, v)$: $u=\alpha {\gamma} {\beta}$ and $v = {\beta'} {\gamma'} {\alpha'}$, for some $\alpha, {\gamma}, {\beta}, {\beta'}, {\gamma'}, {\alpha'} \in \Sigma^+$, and $\textup{sb}(u, v)= ({\gamma}{\beta}, {\beta'}{\gamma'})$, $\textup{sb}(v, u)= ({\gamma'}{\alpha'}, \alpha {\gamma}) $ with  $\textup{lsb}(u, v)= i$, $\textup{lsb}(v, u)= j $ so that $|{\gamma}| = |{\gamma'}| = i+j-n$,  $|\alpha|=|{\alpha'}| = n-i$, and   $|{\beta}|=|{\beta'}| = n-j$. In the rest of Section \ref{subsection2}, we use these notations along with the corresponding geometric elements described in the following.

	Consider the lines $L_1: X+Y = n-i$, $L_2: X+Y = j$ and $L_3: X+Y = n$, and the lattice paths $p = p_O({\alpha {\gamma} {\beta}})$ and $q = p_O({{\alpha'}^R {\gamma'}^R {\beta'}^R})$ from the origin $O = (0, 0)$ in $XY$-plane (For instance, see Fig. \ref{yytr}). Suppose $p$ intersects $L_1$, $L_2$ and $L_3$ at the points $A_{l_1}$, $A_k$, and $A_{m}$, respectively. Note that, since $\alpha {\gamma} \sim_{\textup{abl}} {\alpha'}^R{\gamma'}^R$, both the paths $p$ and $q$ intersect the line $L_2$ at the same point $A_k$. Suppose $q$ intersects $L_1$  and $L_3$ at the points $A_{l_2}$  and $A_{m'}$, respectively.  
	
	Note that the path $p$ has three parts, say $p_1, p_2$ and $p_3$, corresponding to $\alpha, \gamma$, and ${\beta}$, respectively, so that $p = p_1p_2p_3$. Also, the path $q$ has three parts, say $q_1, q_2$ and $q_3$, corresponding to ${\alpha'}^R$,  ${\gamma'}^R$, and ${\beta'}^R$, respectively, so that $q = q_1q_2q_3$. 
	
	Let $A_{l_1} = (l_1, n-i-l_1)$, $A_{l_2} = (l_2, n-i-l_2)$, $A_k = (k, j-k)$, $A_m = (m, n-m)$ and $A_{m'} = (m', n-m')$, where $l_1 \neq l_2$, $0 \leq l_1, l_2 \leq n-i$. 
	Observe that 
	$m' = l_2 + |{\gamma'}^R {\beta'}^R|_a$, 
	$n-m' = n-i-l_2 + |{\gamma'}^R {\beta'}^R|_b$,  
	$m = l_1 + |{\gamma} {\beta}|_a$, and
	$n-m = n- i- l_1 + |{\gamma} {\beta}|_b$. 
	Since ${\gamma} {\beta} \sim_{abl} {\gamma'}^R {\beta'}^R$, 
	$X(A_{m}) - X(A_{m'}) = m - m' = l_1 - l_2 = X(A_{l_1}) - X(A_{l_2})$ and 
	$Y(A_{m'}) - Y(A_{m}) =  Y(A_{l_2}) - Y(A_{l_1})$.
	Thus, the lattice path $p_{A_{l_1}}({{\gamma'}^R {\beta'}^R})$ intersects the line $L_3$ at $A_m$. 
	Suppose $p_{A_{l_1}}({{\gamma'}^R {\beta'}^R})$ intersects $L_2$ at $A_{k'}$, say $(k', j-k')$. Let $q_2' = p_{A_{l_1}}({\gamma'}^R)$ and $q_3' = p_{A_{k'}}({\beta'}^R)$.
	Since ${\gamma} {\beta} \sim_{abl} {\beta'} {\gamma'}$ and $({\gamma} {\beta}, {\beta'} {\gamma'})$ is an MAU pair, 
	$p_2p_3 \cap q_2'q_3' = \{A_{l_1}, A_m \}$.
	Similarly, $p_1p_2 \cap q_1q_2 = \{O, A_k \}$.
	Now by Remark \ref{bijec}, the number of pairs of words $(\alpha {\gamma} {\beta},  {\alpha'}^R {\gamma'}^R {\beta'}^R)$ equals the number of pairs of lattice paths $(p, q)$.
	To count the number of pairs of lattice paths $(p, q)$, we first observe the following:
	\begin{itemize}
		\item  for each pair of paths $(p, q)$, there exist pairs  $(p_1p_2, q_1q_2)$ and $(p_2p_3, q_2'q_3')$ such that 
		$p_1p_2 \cap q_1q_2 = \{O, A_k\}$ and  
		$ p_2p_3 \cap q_2'q_3' = \{A_{l_1}, A_m\}$.
		\item for pairs of paths $(p_1p_2, q_1q_2)$ and $(p_2p_3, q_2'q_3')$ such that 
		$p_1p_2 \cap q_1q_2 = \{O, A_k\}$ and  
		$ p_2p_3 \cap q_2'q_3' = \{A_{l_1}, A_m\}$, 
		there exists a pair of paths $(p, q)$ such that 
		$p_1p_2$ is an initial part of $p$, $p_2p_3$ is a terminal part of $p$,
		$q_1q_2$ is an initial part of $q$, and  
		$ q_2q_3$ is a terminal part of $q$. 
	\end{itemize}
	Thus, $\mathcal{M}^{\textup{o}}(n)$ is the number of pairs of paths $(p_1p_2, q_1q_2)$ and $(p_2p_3, q_2'q_3')$ such that 
	$p_1p_2 \cap q_1q_2 = \{O, A_k\}$ and  
	$ p_2p_3 \cap q_2'q_3' = \{A_{l_1}, A_m\}$.
	 
	First consider $l_1 > l_2$   (see Fig. \ref{yytr}).
	Since we can move only horizontally or vertically along a lattice path,
	$l_1 \leq k \leq l_2 + i+j-n$ and $ k' \leq m \leq k +n-j$. Now we calculate the following:	
	(i) The number of pairs of paths $(p_1, q_1)$ such that $p_1 \cap q_1 = \{ O\}$, i.e., $N_O^{A_{l_1}, A_{l_2}}$,  (ii)
	the number of triplets of paths $(q_2, p_2, q_2')$  such that $q_2 \cap p_2 = \{A_k\}$, and $p_2 \cap q_2' = \{A_{l_1} \}$, i.e., ${}_{A_{l_2}}{{N}}_{A_{l_1}}^{A_k} $, and (iii) the number of pairs of paths $(p_3, q_3')$ such that $p_3 \cap q_3' = \{A_m \}$, i.e., $N_{A_{k}, A_{k'}}^{A_m}$.
	
	Then the number $ (N_O^{A_{l_1}, A_{l_2}}) ({}_{A_{l_2}}{{N}}_{A_{l_1}}^{A_k}) (N_{A_{k}, A_{k'}}^{A_m})$ gives us the required number of paths  for fixed $i, j, l_1, l_2, k, m$. Thus, for fixed $i, j$, and for all $0 \leq l_2 < l_1 \leq n-i$, the number of required paths is
	\begin{center}
		\small
		$S = \displaystyle \sum_{l_2} \sum_{l_1 > l_2} \sum_{k=l_1}^{(l_2 + i +j-n)} \sum_{m= k'}^{(k +n-j)} (N_O^{A_{l_1}, A_{l_2}}) ({}_{A_{l_2}}{{N}}_{A_{l_1}}^{A_k}) (N_{A_{k}, A_{k'}}^{A_m}). $
	\end{center}
	For fixed $i, j$, when $l_1 < l_2$, the required number of paths is also $S$ by  symmetry. Hence,  
	$\displaystyle \mathcal{M}^{\text o}(n)= 2\sum_{j} \sum_{i} \sum_{l_2} \sum_{l_1>l_2} \sum_{k=l_1}^{(l_2 + i +j-n)} \sum_{m= k'}^{(k +n-j)} (N_O^{A_{l_1}, A_{l_2}}) ({}_{A_{l_2}}{{N}}_{A_{l_1}}^{A_k}) (N_{A_{k}, A_{k'}}^{A_m}).$
	
	\subsubsection{Computation} We first compute $N_O^{A_{l_1}, A_{l_2}} $ and $ N_{A_{k}, A_{k'}}^{A_m}$. Then 
	we compute $\mathcal{M}^{\textup{o}}(n)$ for each of cases $|{\gamma}|=1$, $|{\gamma}|=2$, and $|{\gamma}|\geq 3$. In these calculations, we need the points $E_1 = (l_1, n-i-l_1+1)$, the one just above $A_{l_1}$, and $E_2 = (k, j-k-1)$, the one just below $A_k$ (see Fig. \ref{fig3n1} for reference).


	We directly obtain the values of $N_{O}^{A_{l_1}, A_{l_2}}$ and $N_{A_k, A_{k'}}^{A_{m}}$ from \cite[Theorem 2]{gessel1996counting}, as stated in the following proposition.
	
	\begin{proposition}\label{niwe}
		For $l_1 - l_2 \geq 1$,  
		\begin{center}
			$N_{O}^{A_{l_1}, A_{l_2}} =   \frac{l_1 - l_2}{n-i} {n-i \choose l_1} {n-i \choose l_2}\; \text{ and }\; N_{A_k, A_{k'}}^{A_{m}} = \frac{l_1 - l_2}{n-j} {n-j \choose m-k} {n-j \choose m-k'}.$
		\end{center}
	\end{proposition}
	
	\begin{lemma}\label{prop120}
		$||{\alpha}|_c-|{\alpha'}|_c|=1 ~\forall ~c \in \{a, b\}$  
		iff $i+j-n =1$.
	\end{lemma}
	
	\begin{proof}
		Suppose $||{\alpha}|_c-|{\alpha'}|_c|= 1$ $\forall c \in \{a, b\}$.
		Since ${\alpha} {\gamma} \sim_{\textup{abl}} {\gamma'} {\alpha'} $, we have 
		$||{\gamma'}|_c-|{\gamma}|_c| = 1~ \forall c \in \{a, b\}$. Without loss of generality, suppose $|{\alpha}|_a - |{\alpha'}|_a = 1$. If possible, let $i+j-n =|{\gamma}|=|{\gamma'}| \geq 2$. Then we have the following cases: 
			
			\textit{Case 1:} Let ${\gamma} = x a$ and ${\gamma'}^R = y b$ for some $x, y \in \Sigma^+$. Then it can be observed that (see Figure \ref{caseonelemma}) the paths $p_O(\alpha\gamma)$ and $p_O({\alpha'}^R{\gamma'}^R)$ intersect before $A_k$; which is a contradiction as $|\alpha\gamma|$ is shortest external abelian-border of $(u, v)$.

			
			\textit{Case 2:} Let ${\gamma} = x b$ and ${\gamma'}^R = y a$ for some $x, y \in \Sigma^+$.
			Now, ${\gamma} {\beta} \sim_{\textup{abl}} {\beta'} {\gamma'}$  implies $|{\beta}|_a - |{\beta'}|_a  = 1$. Then $|b{\beta}|_a= |{\beta}|_a =1+ |{\beta'}|_a = |{\beta'} a|_a$ and $(u, v)  = ({\alpha} x b {\beta}, {\beta'} a y^R {\alpha'})$. 
			Thus $b {\beta} \sim_{\textup{abl}} {\beta'} a$, i.e., $\textup{lsb}(u, v) \leq |b {\beta}| <|\gamma {\beta}|$, a contradiction.


			
			\textit{Case 3:} Let ${\gamma} = x a_1$ and ${\gamma'}^R = y a_1$ for some $x, y \in \Sigma^+$, $a_1 \in \{a, b\}$.
			Then, ${\alpha} {\gamma} \sim_{\textup{abl}} {\gamma'} {\alpha'} $  implies ${\alpha} x \sim_{\textup{abl}} y^R {\alpha'}$ which is a contradiction as $\textup{lsb}(v, u)=|\alpha \gamma|$.
			
		
		Therefore, our assumption $i+j-n \geq 2$ was wrong. Hence as $i+j-n \geq 1$, $i+j-n =1$.
		
		Conversely, let $i+j-n=1$, i.e., $|{\gamma}| = |{\gamma'}| =1$. Since ${\gamma} \neq {\gamma'}$, consider ${\gamma}=a_1$ and ${\gamma'}=b_1$ where $a_1 \neq b_1$, and $a_1, b_1 \in \{a, b\}$.
		Now, ${\alpha} {\gamma} \sim_{\textup{abl}} {\gamma'} {\alpha'}$ implies   $|{\alpha} a_1|_{a_1} =  |b_1 {\alpha'}|_{a_1}$, i.e., 
		$|{\alpha'}|_{a_1} - |{\alpha}|_{a_1} = 1$.
		Similarly, 
		$|{\alpha} a_1|_{b_1} =  |b_1 {\alpha'}|_{b_1}$ implies $|{\alpha}|_{b_1} -  | {\alpha'}|_{b_1} =  1$.
		Therefore  $||{\alpha}|_c -|{\alpha'}|_c|=1 ~ \forall c \in \{a, b\}$.
	\end{proof}

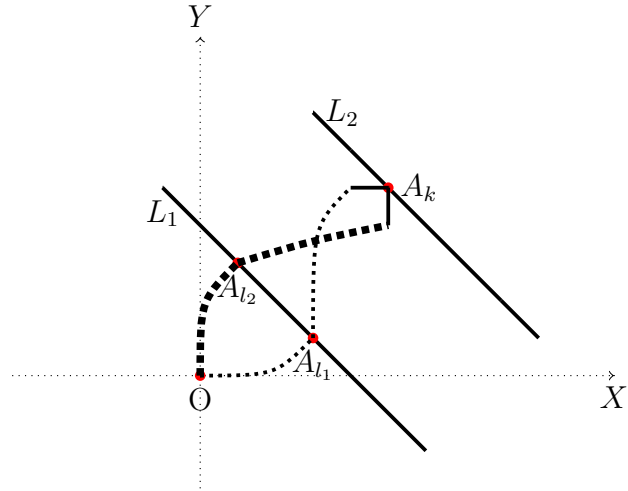
\begin{figure}[h]
				\centering
				\begin{tikzpicture}
                \draw[dotted, ->, line width=0.15mm] (-2,1.5) -- (6,1.5) node[midway, above] {};
				\draw[line width=.1mm] (6,1.5) -- (6,1.5) node[midway,below] {$X$};
				\draw[dotted, ->, line width=0.15mm] (0.5,0) -- (0.5,6) node[midway, above] {}; 
				\draw[line width=.1mm] (0.5,6) -- (0.5,6) node[midway, above] {$Y$};
					
					\draw[line width=.5mm]  (0,4) -- (1,3) node[below] {};
                    \draw[line width=.5mm]  (1,3) -- (2,2) node[below] {};
                    \draw[line width=.5mm]  (2,2) -- (3.5,0.5) node[below] {};
					\draw[line width=.5mm]  (1,3) -- (1,3) node[below] {$A_{l_2}$};
					\fill[red] (1,3) circle (2pt);
					\draw[line width=.5mm]  (2,2) -- (2,2) node[below] {$A_{l_1}$};
					\fill[red] (2,2) circle (2pt);
					\draw[line width=.5mm]  (0,4) -- (0,4) node[below] {$L_1$};

					\draw[line width=.5mm]  (2,5) -- (2.5,4.5) node[below] {};
                    \draw[line width=.5mm]  (2.5,4.5) -- (4.5,2.5) node[below] {};
                    \draw[line width=.5mm]  (4.5,2.5) -- (5,2) node[below] {};
					\fill[red] (3,4) circle (2pt);
					\draw[line width=.5mm]  (3,4) -- (3,4) node[right] {$A_{k}$};
					\draw[line width=.5mm]  (2,5) -- (2,5) node[below,right] {$L_2$};
										
					\draw[line width=.5mm]  (2.5,4) -- (3,4) node[midway, above] {};
					\draw[line width=.5mm]  (3,3.5) -- (3,4) node[midway, right] {};
					
                    \draw[dotted, line width=0.5mm] (2,2) .. controls (2,3.5) .. (2.5,4) node[midway, right] {};           
                    \draw[dotted, line width=1mm] (1,3) .. controls (2,3.3) .. (3,3.5) node[midway, right] {};
                  
					\fill[red] (0.5,1.5) circle (2pt);
					\draw[line width=.5mm]  (0.5,1.5) -- (0.5,1.5) node[below] {O};
					\draw[dotted, line width=0.5mm] (0.5,1.5) .. controls (1.5,1.5) .. (2,2) node[midway, right] {};
					\draw[dotted, line width=1mm] (0.5,1.5) .. controls (0.5,2.5) .. (1,3) node[midway, right] {};														
				\end{tikzpicture}
				\caption{Illustration of proof of Case $1$ of Lemma \ref{prop120}. Here $p_O(\alpha \gamma)$ is thin dotted line and $p_O(\alpha'^R \gamma'^R)$ is the thick dotted line.}
				\label{caseonelemma}
			\end{figure}

	\begin{corollary}
		\begin{enumerate}
			\item $||{\alpha}|_c-|{\alpha'}|_c|=1~\forall c \in \{a, b\} \Leftrightarrow ||{\beta}|_c-|{\beta'}|_c|=1~\forall c \in \{a, b\}$.
			\item $||{\beta}|_c-|{\beta'}|_c|=1~  \forall c \in \{a, b\}  \Leftrightarrow i+j-n =1$.
		\end{enumerate}
	\end{corollary}	
	
	\begin{remark}
		If $i+j = 2n-2$, i.e., $i=j=n-1$, then $|\alpha| = |\gamma| = |\beta| = 1$ so that $n=3$ and $(u,v) \in \{ (aba, bab), (bab, aba)\}$.
	\end{remark}

	\begin{lemma}\label{lem1}
		For $|\gamma| \ge 2$, if $(q_2, p_2, q_2') \in {}_{A_{l_2}}{\mathcal{N}}_{A_{l_1}}^{A_k} $, then $w_{A_{l_2}}^{A_k}(q_2) = aza$ for some $z \in \Sigma^*$ and consequently, $w_{A_{l_1}}^{A_k}(p_2) = bz'b$ for some $z' \in \Sigma^*$. 
	\end{lemma}
	
	\begin{proof}
		Let  $w_{A_{l_2}}^{A_k}(q_2)=w_{_{q_2}}$ and $w_{A_{l_1}}^{A_k}(p_2)=w_{_{p_2}}$. Then $|w_{_{q_2}}|=|w_{_{p_2}}|=|\gamma|$.
		
		\textit{Case 1:} Let $b \in \text{Suff}(w_{_{q_2}})$.
		Then $q_2$ must pass through the point $E_2$. 
		Since $q_2 \cap p_2 = \{A_k\}$,  $p_2$ reaches $A_k$ through $(k-1, j-k)$. Then as $l_1 > l_2$ and $j-k \geq Y(A_{l_2}) > Y(A_{l_1})$, $q_2$ and $p_2$ must intersect before $A_k$, which is a contradiction.

		\textit{Case 2:} Let $b \in \text{Pref}(w_{_{q_2}})$.
		Then as $w_{_{q_2}}= w_{A_{l_1}}(q_2')$, $q_2'$ pass through $E_1$. Since $p_2 \cap q_2' = \{A_{l_1}\}$, $p_2$ must pass through the point $(l_1 +1, n-i-l_1)$. 
		Since $Y(A_k) - Y(A_{k'}) =Y(A_{l_2}) -Y(A_{l_1})>0$, path $q_2'$ and $p_2$ must intersect after $A_{l_1}$, which is a contradiction.
		
		Hence, $w_{_{q_2}}$ is of the form $aza$ for some $z \in \Sigma^*$. Since $p_2 \cap q_2 =\{ A_k\}$ and $p_2 \cap q_2' =\{A_{l_1}\}$, $w_{_{p_2}}$ must be of the form $bz'b$ for some $z' \in \Sigma^*$. 
	\end{proof}
	
	\begin{remark}\label{imprem231}
		By Lemma \ref{lem1}, we have $w_{A_{l_2}}^{A_k}(q_2) =a^{n_1} b^{n_2} \cdots a^{n_{2r-1}}$ and $w_{A_{l_1}}^{A_k}(p_2) = b^{m_1}a^{m_2} \cdots b^{m_{2r'-1}}$ for some $r, r', n_1, \cdots n_{2r-1}, m_1, \cdots m_{2r'-1} \in \mathds{N}$.
	\end{remark}

	
	\begin{proposition}\label{Inticase}
		For $1 \leq|\gamma| \leq 2$, the values of $\mathcal{M}^{\textup{o}}(n)$  are as per the following:
		\begin{itemize}
         \small
			\item When $|\gamma|=1$, then $\mathcal{M}^{\textup{o}}(n)=$   
			$$2 \sum_{i=2}^{n-1} \sum_{l_1=1}^{n-i}  \sum_{m=l_1+1}^{i+l_1-1} 
			\frac{1}{(i-1)(n-i)} {n-i \choose l_1} {n-i \choose l_1-1}  {i-1 \choose m-l_1} {i-1 \choose m-l_1-1}.$$

			\item   When $|\gamma|=2$, then $\mathcal{M}^{\textup{o}}(n)=$ 
			$$2 \sum_{i=3}^{n-2} \sum_{l_1=2}^{n-i}  \sum_{m=l_1+2}^{i+l_1-2}   
			\frac{4}{(i-2)(n-i)} {n-i \choose l_1} {n-i \choose l_1-2}  {i-2 \choose m-l_1} {i-2 \choose m-l_1-2}.$$
			
			


		\end{itemize}
		
	\end{proposition}

	\begin{proof}
		Since we can move only horizontally or vertically along a lattice path,
		$\max\{l_1,l_2\} \leq k \leq \min\{l_1, l_2\} + |\gamma|$ and  
		$ \max\{k, k' \} \leq m \leq \min\{k, k'\} +n-j$ (see Fig. \ref{yytr}).
		We now compute the value of ${}_{A_{l_2}}{{N}}_{A_{l_1}}^{A_k}$ for  $1 \leq|\gamma| \leq 4$.
		
		\begin{itemize}
		    \item 
		Let $|{\gamma}|=1$, i.e., $|{\gamma'}|=i+j-n =1$.
		Then by Lemma \ref{prop120}, $||{\alpha}|_c-|{\alpha'}|_c|=1 ~\forall c \in \{a, b\}$. Let $|{\alpha}|_a-|{\alpha'}|_a=1$. Then $l_1 - l_2=1$, $A_{l_2} = (l_1 -1, n-i-l_1+1)$ and  $l_1 \leq k \leq l_1 $, i.e., $ k = l_1$ (see Fig. \ref{fig12io}). Then $Y(A_k)=j-k=n-i+1-l_1$ and $A_k = (l_1, n-i-l_1+1)$. This implies  ${\gamma} = b$ and ${\gamma'}^R=a$.
			Then $A_{k'} = (l_1 +1, n-i-l_1)$. This implies ${}_{A_{l_2}}{{N}}_{A_{l_1}}^{A_k} =1$. Now  $ l_1+1 \leq m \leq l_1 +n-j$, i.e.,  $ l_1+1 \leq m \leq l_1 +i-1$. 
			Thus, for $|\gamma|=1 $, the value of $\mathcal{M}^{\textup{o}}(n)$  is
			\begin{center}
				\small
				$ 2 \displaystyle\sum_{i=2}^{n-1} \sum_{l_1=1}^{n-i}  \sum_{m=l_1+1}^{i+l_1-1}  
				\frac{1}{(i-1)(n-i)} {n-i \choose l_1} {n-i \choose l_1-1}  {i-1 \choose m-l_1} {i-1 \choose m-l_1-1}$.
			\end{center}
            Here $ i \ge 2$ as  $|\gamma|=1,|\beta|\geq 1$.

			\begin{figure}[htbp]
		\centering
		\begin{minipage}[b]{0.4\textwidth}
			\centering
			\includegraphics[width=\linewidth]{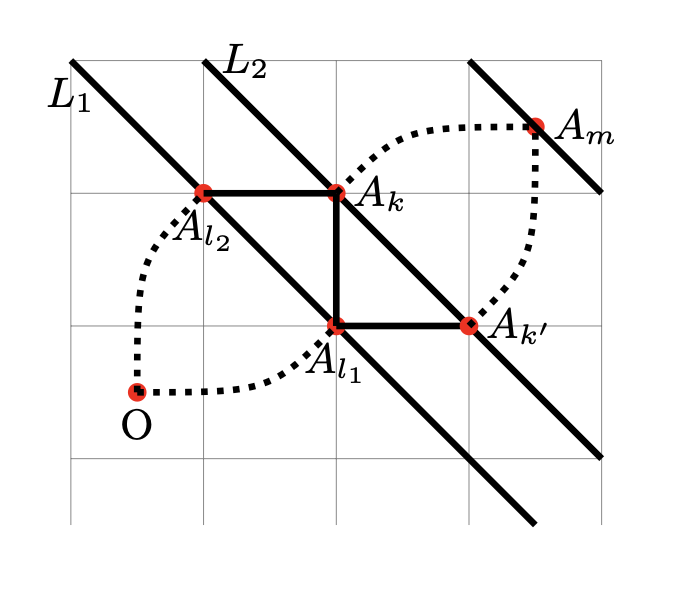}
			\caption{Illustration to the proof of Proposition \ref{Inticase} when $|\gamma|=1$.}
			\label{fig12io}
		\end{minipage}
		\hspace{1cm}
		\begin{minipage}[b]{0.4\textwidth}
			\centering
			\includegraphics[width=\linewidth]{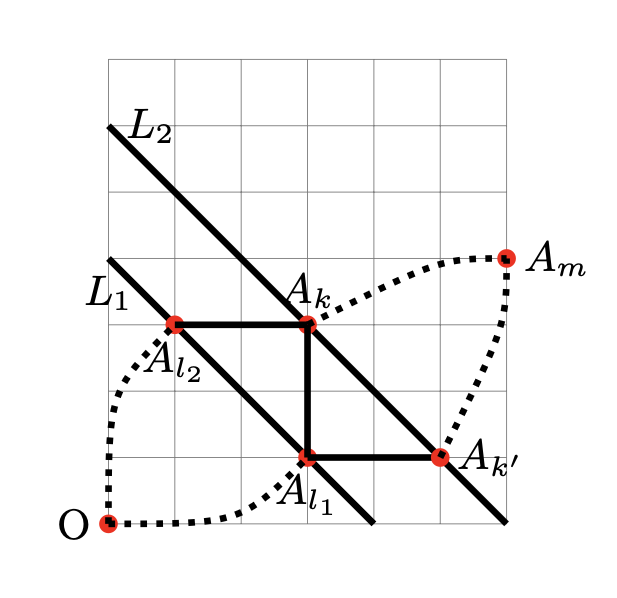}
			\caption{Illustration to the proof of Proposition \ref{Inticase} when $|\gamma|=2$.}
			\label{fi5sw}
		\end{minipage}
	\end{figure}

			\item Let $|{\gamma}| =2$, i.e., $|{\gamma'}|=i+j-n =2$. Then
			by Lemma \ref{prop120}, $||{\alpha}|_c-|{\alpha'}|_c| \geq 2 ~\forall c \in \{a, b\}$.
			Let $|{\alpha}|_a-|{\alpha'}|_a \geq 2$. Then $l_1-l_2\geq 2$. 
			Now, $l_1 \leq k \leq l_2+2$. This implies $l_1-l_2 \leq 2$. Thus $l_1-l_2=2$. Thus $A_{l_2} = (l_1-2, n-i-l_1+2)$.
			Now $l_1 \leq k \leq l_1 -2+2$, i.e., $k=l_1$ (see Fig. \ref{fi5sw}). Then $Y(A_k) = j-k= n-i+2 -l_1$ and $A_k = (k, n-i -l_1 + 2)$. This implies $\gamma=bb$ and ${\gamma'}^R=aa$. Then $A_{k'} = (l_1+2, n-i-l_1)$. Then  ${}_{A_{l_2}}{{N}}_{A_{l_1}}^{A_k}=1$.
			Now $k' \leq m \leq k+n-j$, i.e., $l_1 +2 \leq m \leq l_1 +i-2$. Thus the value of $\mathcal{M}^{\textup{o}}(n)$ for $|\gamma|=2 $ is
			\begin{center}
				\small
				$2\displaystyle \sum_{i=4}^{n-2} \sum_{l_1=2}^{n-i}  \sum_{m=l_1+2}^{i+l_1-2}  
				\frac{4}{(i-2)(n-i)} {n-i \choose l_1} {n-i \choose l_1-2}  {i-2 \choose m-l_1} {i-2 \choose m-l_1-2}$.
			\end{center}
            Here $i \geq 4$ as $|\gamma|=2$, $|\beta|\geq 2$.

		\end{itemize}
	\end{proof}

	We have verified Proposition \ref{Inticase} through a computer program and the results are given in Table \ref{tab:sample_table1213}.
	\begin{table}[h]
		\small
		\centering
		\begin{tabular}{|c|c|c|c|c|c|c|c|c|c|c|c|c|}
			\hline
			$n$ & 1 & 2 & 3 & 4 & 5 & 6 & 7 & 8 & 9 & 10 & 11 & 12 \\
			\hline
			$\mathcal{M}^{\textup{o}}(n)$ for $|\gamma|=1$   & 0 & 0 & 2 & 8 & 28 & 96 & 330 & 1144 & 4004 & 14144 & 50388 & 180880   \\
			\hline
			$\mathcal{M}^{\textup{o}}(n)$ for $|\gamma|=2$  & 0 & 0 & 0 & 0 & 0 & 2 & 16 & 88 & 416 & 1820 & 7616 & 31008   \\
			\hline
		\end{tabular}
		\caption{Values of $\mathcal{M}^{\textup{o}}(n)$, for $1 \le |\gamma| \le 2$  and $1 \leq n \leq 12$.}
		\label{tab:sample_table1213}
	\end{table}
	

	 \begin{figure}[h]
	\centering
	\begin{tikzpicture}
		\draw[step=0.5cm,gray,line width=0.1pt] (-1.5,-1) grid (11.5,10);
		
		\draw[line width=.4mm]  (7.5,9.5) -- (7.5,9.5) node[midway,right] {$A_k$};
        \draw[line width=.4mm]  (11,6) -- (11,6) node[midway,above] {$A_{k'}$};
        \draw[dotted,line width=.4mm]  (7.7,3.4) -- (7.7,3.4) node[above] {$Q$};
        \draw[dotted,line width=.4mm]  (7.8,5) -- (7.8,5) node[above] {$Q'$};
        \draw[dotted,line width=.4mm]  (5.4,7) -- (5.4,7) node[below] {$R$};
        \draw[dotted,line width=.4mm]  (9,3.5) -- (9,3.5) node[below] {$R'$};

		\draw[line width=.4mm]  (2.5,-0.5) -- (1,1) node[below] {};
		\draw[dotted, line width=.4mm]  (1,1) -- (0.5,1.5) node[below] {};
		\draw[line width=.4mm]  (0.5,1.5) -- (-1,3) node[below] {};
		\draw[line width=.4mm]  (2.5,-0.5) -- (2.5,-0.5) node[below,midway] {$A_{l_1}$};
		\fill[black] (2.5,-0.5) circle (2.5pt);
		\fill[black] (2.5,0) circle (2.5pt);
		\draw[line width=.4mm]  (2.5,0) -- (2.5,0) node[right] {$E_1$};
		\draw[line width=.4mm]  (-1,3) -- (-1,3) node[left,midway] {$A_{l_2}$};
		\fill[black] (-1,3) circle (2.5pt);

		\draw[line width=.4mm]  (-1,3) -- (-0.5,3) node[below] {};
		\draw[dotted,line width=.4mm]  (-0.5,3) -- (0.5,3) node[midway,above] {$a^{n_1}$};
		\draw[line width=.4mm]  (0.5,3) -- (1,3) node[below] {};
		
		\draw[line width=.4mm]  (1,3) -- (1,3.5) node[below] {};
		\draw[dotted,line width=.4mm]  (1,3.5) -- (1,4) node[midway,left] {$b^{n_2}$};
		\draw[line width=.4mm]  (1,4) -- (1,4.5) node[below] {};
		
		\draw[dotted, line width=.4mm] (1,4.5) .. controls (1.5,4.5) .. (1.5,5.5);
		
		\draw[line width=.4mm]  (1.5,5.5) -- (2,5.5) node[below] {};
		\draw[dotted,line width=.4mm]  (2,5.5) -- (2.5,5.5) node[midway,above] {$a^{n_{2l-1}}$};
		\draw[line width=.4mm]  (2.5,5.5) -- (3,5.5) node[below] {};
        \draw[line width=.4mm]  (3,5.5) -- (3,5.5) node[right] {$B_2$};
		
		\draw[line width=.4mm]  (3,5.5) -- (3,6) node[below] {};
		\draw[dotted,line width=.4mm]  (3,6) -- (3,6.5) node[midway,left] {$b^{n_{2l}}$};
		\draw[line width=.4mm]  (3,6.5) -- (3,7) node[below] {}; 
        \draw[line width=.4mm]  (3,7) -- (3,7) node[above] {$B_4$}; 
		
		\draw[line width=.4mm]  (3,7) -- (3.5,7) node[below] {};
		\draw[dotted,line width=.4mm]  (3.5,7) -- (5,7) node[midway,above] {$a^{n_{2l+1}}$};
		\draw[line width=.4mm]  (5,7) -- (5.5,7) node[below] {}; 
		
		\draw[dotted, line width=.4mm] (5.5,7) .. controls (5.5,7.5) .. (6,8);
		
		\draw[line width=.4mm]  (6,8) -- (6,8.5) node[below] {};
		\draw[dotted,line width=.4mm]  (6,8.5) -- (6,9) node[midway,left] {$b^{n_{2r-2}}$};
		\draw[line width=.4mm]  (6,9) -- (6,9.5) node[below] {}; 
		
		\draw[line width=.4mm]  (6,9.5) -- (6.5,9.5) node[below] {};
		\draw[dotted,line width=.4mm]  (6.5,9.5) -- (7,9.5) node[midway,above] {$a^{n_{2r-1}}$};
		\draw[line width=.4mm]  (7,9.5) -- (7.5,9.5) node[below] {};
		\fill[black] (7.5,9.5) circle (2.5pt);
		
		\fill[black] (7.5,9) circle (2.5pt);
		\draw[line width=.4mm]  (7.5,9) -- (7.5,9) node[midway,right] {$E_2$};
		
		\draw[->, line width=.4mm]  (5.8,9.3)  -- (4.7,9.3) node[below] {};
		\draw[ line width=.5mm]  (4.7,9.3)  -- (4.7,9.3) node[left] {path $q_2$ };

		\draw[line width=.4mm]  (2.5,-0.5) -- (3,-0.5) node[below] {};
		\draw[dotted,line width=.4mm]  (3,-0.5) -- (4,-0.5) node[midway,below] {$a^{n_1}$};
		\draw[line width=.4mm]  (4,-0.5) -- (4.5,-0.5) node[below] {};
		
		\draw[line width=.4mm]  (4.5,-0.5) -- (4.5,0) node[below] {};
		\draw[dotted,line width=.4mm]  (4.5,0) -- (4.5,0.5) node[midway,right] {$b^{n_2}$};
		\draw[line width=.4mm]  (4.5,0.5) -- (4.5,1) node[below] {};
		
		\draw[dotted, line width=.4mm] (4.5,1) .. controls (5,1) .. (5,2);
		
		\draw[line width=.4mm]  (5,2) -- (5.5,2) node[below] {};
		\draw[dotted,line width=.4mm]  (5.5,2) -- (6,2) node[midway,below] {$a^{n_{2l-1}}$};
		\draw[line width=.4mm]  (6,2) -- (6.5,2) node[below] {}; 
        \draw[line width=.4mm]  (6.5,2) -- (6.5,2) node[right] {$B_1$}; 
		
		\draw[line width=.4mm]  (6.5,2) -- (6.5,2.5) node[below] {};
		\draw[dotted,line width=.4mm]  (6.5,2.5) -- (6.5,3) node[midway,right] {$b^{n_{2l}}$};
		\draw[line width=.4mm]  (6.5,3) -- (6.5,3.5) node[below] {}; 
		\draw[line width=.4mm]  (6.6,3.5) -- (6.6,3.5) node[right,above] {$B_3$};
        
		\draw[line width=.4mm]  (6.5,3.5) -- (7,3.5) node[below] {};
		\draw[dotted,line width=.4mm]  (7,3.5) -- (8.5,3.5) node[midway,below] {$a^{n_{2l+1}}$};
		\draw[line width=.4mm]  (8.5,3.5) -- (9,3.5) node[below] {};
		
		\draw[dotted, line width=.5mm] (9,3.5) .. controls (9,4) .. (9.5,4.5);

		\draw[line width=.4mm]  (9.5,4.5) -- (9.5,5) node[below] {};
		\draw[dotted,line width=.4mm]  (9.5,5) -- (9.5,5.5) node[midway,right] {$b^{n_{2r-2}}$};
		\draw[line width=.4mm]  (9.5,5.5) -- (9.5,6) node[below] {};
		
		\draw[line width=.4mm]  (9.5,6) -- (10,6) node[below] {};
		\draw[dotted,line width=.4mm]  (10,6) -- (10.5,6) node[midway,below] {$a^{n_{2r-1}}$};
		\draw[line width=.4mm]  (10.5,6) -- (11,6) node[below] {};
		\fill[black] (11,6) circle (2.5pt);

		\draw[->, line width=.4mm]  (8.5,2.5)  -- (8.5,3.3) node[below] {};
		\draw[ line width=.5mm]  (8.5,2.4)  -- (8.5,2.4) node[below] {path $q_2'$ };

		\draw[line width=.4mm]  (4.5,-0.5) -- (3,1) node[below] {};
		\draw[dotted,line width=.4mm]  (3,1) -- (2.5,1.5) node[below] {};
		\draw[line width=.4mm]  (2.5,1.5) -- (1,3) node[below] {};
		\draw[line width=.4mm]  (4.5,1) -- (3,2.5) node[below] {};
		\draw[dotted,line width=.4mm]  (3,2.5) -- (2.5,3) node[below] {};
		\draw[line width=.4mm]  (2.5,3) -- (1,4.5) node[below] {};
		\draw[line width=.4mm]  (5,2) -- (3.5,3.5) node[below] {};
		\draw[dotted,line width=.4mm]  (3.5,3.5) -- (3,4) node[below] {};
		\draw[line width=.4mm]  (3,4) -- (1.5,5.5) node[below] {};
		\draw[line width=.4mm]  (7.2,1.3) -- (5,3.5) node[below] {};
		\draw[dotted,line width=.4mm]  (5,3.5) -- (4.5,4) node[below] {};
		\draw[line width=.4mm]  (4.5,4) -- (3,5.5) node[below] {};
		\draw[->, line width=.4mm]  (7.2,1.4)  -- (7.7,1.4) node[right] {$X+Y= n-i+ \displaystyle \sum_{i=1}^{2l-1} n_i$};
		\draw[line width=.4mm]  (6.5,3.5) -- (5,5) node[below] {};
		\draw[dotted,line width=.4mm]  (5,5) -- (4.5,5.5) node[below] {};
		\draw[line width=.4mm]  (4.5,5.5) -- (2.5,7.5) node[below] {};
		\draw[->, line width=.4mm]  (2.6,7.1)  -- (2,7.1) node[left] {$X+Y= n-i+ \displaystyle \sum_{i=1}^{2l} n_i$};
		\draw[line width=.4mm]  (9,3.5) -- (7,5.5) node[below] {};
		\draw[dotted,line width=.4mm]  (7,5.5) -- (6.5,6) node[below] {};
		\draw[line width=.4mm]  (6.5,6) -- (4.5,8) node[below] {};
		\draw[<-, dotted, line width=0.4mm] (4.4,8) .. controls (3.8,8.3) .. (3.4,8.3) node[left] {$X+Y= n-i+ \displaystyle \sum_{i=1}^{2l+1} n_i$};
		\draw[line width=.4mm]  (9.5,4.5) -- (8,6) node[below] {};
		\draw[dotted,line width=.4mm]  (8,6) -- (7.5,6.5) node[below] {};
		\draw[line width=.4mm]  (7.5,6.5) -- (6,8) node[below] {};
		\draw[line width=.4mm]  (9.5,6) -- (8,7.5) node[below] {};
		\draw[dotted,line width=.4mm]  (8,7.5) -- (7.5,8) node[below] {};
		\draw[line width=.4mm]  (7.5,8) -- (6,9.5) node[below] {};

		\draw[line width=.5mm]  (7.5,9.5) -- (7.5,9) node[below] {};
		\draw[dotted,line width=.5mm]  (7.5,9) -- (7.5,1) node[below] {};
		\draw[line width=.5mm]  (7.5,1) -- (7.5,0.5) node[below] {$x=k$};
		
		\draw[line width=0.4mm] (2.5,-0.5) -- (2.5,0) node[left] {};
		\draw[-, line width=0.4mm] (2.5,0) .. controls (3,0.5) .. (3,1) node[left] {};
		\draw[-, line width=0.4mm] (3,1) .. controls (3.5,1.5) .. (3.5,2) node[left] {};
		\draw[-,dotted, line width=0.4mm] (3.5,2) .. controls (3.5,2.5) .. (4,3) node[left] {};
		\draw[-, line width=0.4mm] (4,3) .. controls (4.5,3.5) .. (4.5,4) node[left] {};
		\draw[-, line width=0.4mm] (4.5,4) .. controls (4.5,4.5) .. (5,5) node[left] {};
		\draw[-, line width=0.4mm] (5,5) .. controls (6,5.4) .. (6,6.5) node[left] {};
		\draw[-,dotted, line width=0.4mm] (6,6.5) .. controls (6.3,6.8) .. (6.5,7.5) node[left] {};
		\draw[-, line width=0.4mm] (6.5,7.5) .. controls (6.5,8) .. (7,8.5) node[left] {};
		\draw[-, line width=0.4mm] (7,8.5) .. controls (7.3,8.7) .. (7.5,9) node[left] {};
		\draw[->,dotted, line width=0.4mm] (0,0.5) .. controls (2.5,1) .. (2.8,0.6) node[below] {};
		\draw[-, line width=0.4mm] (0,0.5) -- (0,0.5) node[left] {path $p_2$};

		\fill[black] (4,0) circle (1.5pt);
		\draw (4,0) circle (0.15cm);
		
		\fill[black] (3.5,0.5) circle (1.5pt);
		\draw (3.5,0.5) circle (0.15cm);
		
		\fill[black] (3,1) circle (1.5pt);
		\draw (3,1) circle (0.15cm);
		
		\fill[black] (2.5,1.5) circle (1.5pt);
		\draw (2.5,1.5) circle (0.15cm);
		
		\fill[black] (2,2) circle (1.5pt);
		\draw (2,2) circle (0.15cm);
		
		\fill[black] (1.5,2.5) circle (1.5pt);
		\draw (1.5,2.5) circle (0.15cm);
		
		\fill[black] (4,1.5) circle (1.5pt);
		\draw (4,1.5) circle (0.15cm);
		
		\fill[black] (3.5,2) circle (1.5pt);
		\draw (3.5,2) circle (0.15cm);
		
		\fill[black] (3,2.5) circle (1.5pt);
		\draw (3,2.5) circle (0.15cm);
		
		\fill[black] (2.5,3) circle (1.5pt);
		\draw (2.5,3) circle (0.15cm);
		
		\fill[black] (2,3.5) circle (1.5pt);
		\draw (2,3.5) circle (0.15cm);
		
		\fill[black] (1.5,4) circle (1.5pt);
		\draw (1.5,4) circle (0.15cm);

		\fill[black] (4.5,2.5) circle (1.5pt);
		\draw (4.5,2.5) circle (0.15cm);
		
		\fill[black] (4,3) circle (1.5pt);
		\draw (4,3) circle (0.15cm);
		
		\fill[black] (3.5,3.5) circle (1.5pt);
		\draw (3.5,3.5) circle (0.15cm);
		
		\fill[black] (3,4) circle (1.5pt);
		\draw (3,4) circle (0.15cm);
		
		\fill[black] (2.5,4.5) circle (1.5pt);
		\draw (2.5,4.5) circle (0.15cm);
		
		\fill[black] (2,5) circle (1.5pt);
		\draw (2,5) circle (0.15cm);
		
		\fill[black] (6,2.5) circle (1.5pt);
		\draw (6,2.5) circle (0.15cm);
		
		\fill[black] (5.5,3) circle (1.5pt);
		\draw (5.5,3) circle (0.15cm);
		
		\fill[black] (5,3.5) circle (1.5pt);
		\draw (5,3.5) circle (0.15cm);
		
		\fill[black] (4.5,4) circle (1.5pt);
		\draw (4.5,4) circle (0.15cm);
		
		\fill[black] (4,4.5) circle (1.5pt);
		\draw (4,4.5) circle (0.15cm);
		
		\fill[black] (3.5,5) circle (1.5pt);
		\draw (3.5,5) circle (0.15cm);
		
		\fill[black] (6,4) circle (1.5pt);
		\draw (6,4) circle (0.15cm);
		
		\fill[black] (5.5,4.5) circle (1.5pt);
		\draw (5.5,4.5) circle (0.15cm);
		
		\fill[black] (5,5) circle (1.5pt);
		\draw (5,5) circle (0.15cm);
		
		\fill[black] (4.5,5.5) circle (1.5pt);
		\draw (4.5,5.5) circle (0.15cm);
		
		\fill[black] (4,6) circle (1.5pt);
		\draw (4,6) circle (0.15cm);
		
		\fill[black] (3.5,6.5) circle (1.5pt);
		\draw (3.5,6.5) circle (0.15cm);

		\fill[black] (7.5,5) circle (1.5pt);
		\draw (7.5,5) circle (0.15cm);
		
		\fill[black] (7,5.5) circle (1.5pt);
		\draw (7,5.5) circle (0.15cm);
		
		\fill[black] (6.5,6) circle (1.5pt);
		\draw (6.5,6) circle (0.15cm);
		
		\fill[black] (6,6.5) circle (1.5pt);
		\draw (6,6.5) circle (0.15cm);
		
		\fill[black] (7.5,6.5) circle (1.5pt);
		\draw (7.5,6.5) circle (0.15cm);
		
		\fill[black] (7,7) circle (1.5pt);
		\draw (7,7) circle (0.15cm);
		
		\fill[black] (6.5,7.5) circle (1.5pt);
		\draw (6.5,7.5) circle (0.15cm);

		\fill[black] (7.5,8) circle (1.5pt);
		\draw (7.5,8) circle (0.15cm);
		
		\fill[black] (7,8.5) circle (1.5pt);
		\draw (7,8.5) circle (0.15cm);
		
		\fill[black] (6.5,9) circle (1.5pt);
		\draw (6.5,9) circle (0.15cm);

	\end{tikzpicture}
	\caption{Illustration to the proof of Lemma \ref{nbve} and Proposition \ref{Prop1}. Here, big dots surrounded by circles denote internal points.}
	\label{fig3n1}
\end{figure}

	For $r\in \mathds{N}$, consider $L'_r: X+Y= n-i + \sum_{t=1}^{r} n_t$ and $\mathtt{IP}_{q{_{_2}},q_2'}(L'_r)$, in short we write $\mathtt{IP}_{q_{_2},q_2'}(r)$. We use these notations only in Lemma \ref{nbve}.
	
	\begin{lemma}\label{nbve}
    Let $w_{A_{l_2}}^{A_k}(q_2)=w_{A_{l_1}}^{A_{k'}}(q_2') =a^{n_{1}} b^{n_{2}} a^{n_{3}}\cdots  a^{n_{2r-1}}$ for some lattice paths $q_2, q_2'$ where  $l_1 - l_2 \geq 2$, $r \geq 2$ and each $n_t \geq 1$. Then for any $1\leq s \leq 2r-2$, the lattice paths between the points of   $\mathtt{IP}_{q_{_2},q_{2}'}(s)$  and $\mathtt{IP}_{q_{_2},q_{2}'}({s+1})$ do not intersect $q_2$ and $q_2'$. 
		Also, lattice paths from the point $E_1$ to $\mathtt{IP}_{q_{_2},q_{2}'}(1)$, and that from the points of $\mathtt{IP}_{q_{_2},q_{2}'}({2r-2})$ to the point $E_2$ do not intersect $q_2$ and $q_2'$. 
	\end{lemma}

	\begin{proof}
		For $1 \leq l \leq {r-1}$, let $L'_{2l-1}$ intersects $q_2'$ and $q_2$ at the points $B_1$ and $B_2$, respectively (see Fig. \ref{fig3n1}). Also, let $L'_{2l}$ intersects $q_2'$ and $q_2$ at the points $B_3$ and $B_4$, respectively (see Fig. \ref{fig3n1}). By geometric argument,  $X(B_2)=X(B_4)$ and $X(B_1)=X(B_3)$.
		If  $(X_1,Y_1)\in \mathtt{IP}_{q_{_2},q_{2}'}(2l-1)$ and  $(X_2,Y_2) \in  \mathtt{IP}_{q_{_2},q_{2}'}(2l)$, then  $X(B_2)+1 \leq X_1 \leq X(B_1)-1$ and $X(B_4)+1 \leq X_2 \leq X(B_3)-1$.  
		If $p'$ is any lattice path between $(X_1, Y_1)$ and $(X_2, Y_2)$ and $C$ is any point on $p'$, then as we can move only horizontally or vertically along a lattice path, $X(B_2)+1 \leq X(C)\leq X(B_1)-1$. Thus $p'$ never intersects $q_2$ and $q_2'$. 
		Therefore any lattice path (if it exists) between $(X_1, Y_1)$ and $(X_2, Y_2)$ do not intersect $q_2'$ and $q_2$. Since $(X_1, Y_1)$ and $(X_2, Y_2)$ are arbitrary, this holds between any two points of $ \mathtt{IP}_{q_{_2},q_{2}'}(2l-1)$ and   $\mathtt{IP}_{q_{_2},q_{2}'}(2l)$. Similar things happen between the points of  $\mathtt{IP}_{q_{_2},q_{2}'}(2l)$ and  $\mathtt{IP}_{q_{_2},q_{2}'}(2l+1)$. This proves the first part of the statement.
		Taking clue from Fig. \ref{fig3n1}, one can prove the remaining parts of the statement. 
	\end{proof}

	We now compute $ {}_{A_{l_2}}{{N}}_{A_{l_1}}^{A_k}$ when $|\gamma|\geq 3$.
    To do that, we first find the coordinates of internal points. 

    	\begin{lemma}\label{internal points}
         Let $w_{A_{l_2}}^{A_k}(q_2)=w_{A_{l_1}}^{A_{k'}}(q_2') =a^{n_{1}} b^{n_{2}} a^{n_{3}}\cdots  a^{n_{2r-1}}$ for some lattice paths $q_2, q_2'$ where  $l_1 - l_2 \geq 2$, $r \geq 2$ and each $n_t \geq 1$. Let $ k = l_1+n_1+n_3+\cdots + n_{2l-1} + n_{2l+1}'$ for some $0 \leq l\leq r-1$ where $n_{2l+1}=n_{2l+1}' + n_{2l+1}''$ with $n_{2l+1}'\geq 0 $ and $n_{2l+1}'' > 0$.
		Then  
            \begin{enumerate}[label=\rm (\roman*)]
            \tiny
                \item For any $r' \leq l$,  
                \begin{itemize}
                    \item  $\mathtt{IP}_{q_{_2},q_2'}(X+Y = n-i+ \sum_{i'=1}^{2r'-1}n_{i'} ~\&~ X=k) = $
		$$ \{  (X(A_{l_2})+ \sum_{i'=1}^{r'}n_{2i'-1} + s_{2r'-1}, Y(A_{l_2})+ \sum_{i'=1}^{r'-1}n_{2i'} - s_{2r'-1} ) : 1 \leq s_{2r'-1} \leq l_1- l_2 -1  \} .$$

                    \item    $\mathtt{IP}_{q_{_2},q_2'}(X+Y = n-i  + \sum_{i'=1}^{2r'}n_{i'}~\&~ X=k) =$ 
		$$ \{ (X(A_{l_2}) + \sum_{i'=1}^{r'}n_{2i'-1}  + t_{2r'}, Y(A_{l_2})+ \sum_{i'=1}^{r'}n_{2i'} - t_{2r'}) : 1 \leq t_{2r'} \leq l_1- l_2 -1  \}.$$ 
                \end{itemize}
                
                \item $\mathtt{IP}_{q_{_2},q_2'}(X+Y = n-i+\sum_{i'=1}^{2l+1}n_{i'}~ \&~ X=k) =$
                $$ \{ (X(A_{l_2}) +\sum_{i'=1}^{l+1}n_{2i'-1} + s_{2l+1}, Y(A_{l_2})+\sum_{i'=1}^{l}n_{2i'} - s_{2l+1}) : 1 \leq s_{2l+1} \leq l_1- l_2 - n_{2l+1}''  \}$$

                \item $\mathtt{IP}_{q_{_2},q_2'}(X+Y = n-i+\sum_{i'=1}^{2l+2}n_{i'} ~\&~ X=k) =$
                $$ \{ (X(A_{l_2}) + \sum_{i'=1}^{l+1}n_{2i'-1} + t_{2l+2}, Y(A_{l_2})+\sum_{i'=1}^{l+1}n_{2i'} - t_{2l+2}) : 1 \leq t_{2l+2} \leq l_1- l_2 - n_{2l+1}''  \}$$

                \item For an odd integer $3 \leq s \leq 2r-2l-3$, 
                \begin{itemize}
                    \item $\mathtt{IP}_{q_{_2},q_2'}(X+Y = n-i+  \sum_{i'=1}^{2l+s}n_{i'}  ~\&~ X=k) = $
            $$\{ (X(A_{l_2})+ \sum_{i'=1}^{\frac{2l+s+1}{2}}n_{2i'-1}  + s_{2l+s}, Y(A_{l_2})+\sum_{i'=1}^{\frac{2l+s-1}{2}}n_{2i'} - s_{2l+s}) : 1 \leq s_{2l+s} \leq l_1- l_2 - n_{2l+1}'' -\sum_{i'=1}^{\frac{s-1}{2}}n_{2l+2i'+1} \}$$

            \item $\mathtt{IP}_{q_{_2},q_2'}(X+Y = n-i+\sum_{i'=1}^{2l+s+1}n_{i'}~\&~ X=k) =$
            $$ \{ (X(A_{l_2})+\sum_{i'=1}^{\frac{2l+s+1}{2}}n_{2i'-1}  + t_{2l+s+1}, Y(A_{l_2})+\sum_{i'=1}^{\frac{2l+s+1}{2}}n_{2i'} - t_{2l+s+1}) : 1 \leq t_{2l+s+1} \leq l_1- l_2 - n_{2l+1}'' -\sum_{i'=1}^{\frac{s-1}{2}}n_{2l+2i'+1} \}.$$  
                \end{itemize}

            \end{enumerate}	 
	   \end{lemma}

       \begin{proof}
       
       Let $w_{A_{l_2}}^{A_k}(q_2)=w_{A_{l_1}}^{A_{k'}}(q_2') =a^{n_{1}} b^{n_{2}} a^{n_{3}}\cdots  a^{n_{2r-1}}$ for some lattice paths $q_2, q_2'$.
       Then by Figure \ref{fig3n1}, 
		\begin{align}
			k &= X(A_{l_2}) + n_{1} + n_{3} + \cdots + n_{2r-1} \label{equq1} \\
			j - k &= Y(A_{l_2}) + n_{2} + n_{4} + \cdots + n_{2r-2}. \label{equq2}
		\end{align}
        Let $k= X(A_{l_1}) + n_{1} + n_{3} + \cdots + n_{2l-1} + n_{2l+1}'$ for some $0 \leq l \leq r-1$ where $n_{2l+1}=n_{2l+1}' + n_{2l+1}''$ with $n_{2l+1}' \geq 0$ and $n_{2l+1}'' > 0$.
        Then from Equation (\ref{equq1}), 
		$ X(A_{l_1}) + n_{1} + n_{3} + \cdots + n_{2l-1} + n_{2l+1}'   =  X(A_{l_2}) + n_{1} + n_{3} + \cdots + n_{2r-1} $.  
		This implies
		\begin{equation}\label{foundl}
			l_1-l_2 = n_{2l+1}'' + n_{2l+3} +\cdots + n_{2r-1}.
		\end{equation}
        Let $r' \leq l$. Then from Figure \ref{fig3n1}, it is straightforward that 
        \begin{itemize}
            \item $\mathtt{IP}_{q_2,q_2'}(X+Y = n-i+ \sum_{i'=1}^{2r'-1}n_{i'} ~\&~ X=k) = $
		$$ \{  (X(A_{l_2})+ \sum_{i'=1}^{r'}n_{2i'-1} + s_{2r'-1}, Y(A_{l_2})+ \sum_{i'=1}^{r'-1}n_{2i'} - s_{2r'-1} ) : 1 \leq s_{2r'-1} \leq l_1- l_2 -1  \} .$$

        \item $\mathtt{IP}_{q_2,q_2'}(X+Y = n-i  + \sum_{i'=1}^{2r'}n_{i'}~\&~ X=k) =$ 
		$$ \{ (X(A_{l_2}) + \sum_{i'=1}^{r'}n_{2i'-1}  + t_{2r'}, Y(A_{l_2})+ \sum_{i'=1}^{r'}n_{2i'} - t_{2r'}) : 1 \leq t_{2r'} \leq l_1- l_2 -1  \}.$$ 
        \end{itemize}
        This proves $(i)$.

        The line $X=k$ intersects the path $q_2$ only at the point $A_k$. 
        From Figure \ref{fig3n1}, it is evident that $(k, Y(A_{l_1}) + n_2+ n_4 +\cdots+ n_{2l})$ is a point of intersection of the path $q_2'$ and line $X=k$. Let us call this point $Q$.
		If $Q'$ is the point of intersection of lines $X=k$ and $X+Y = n-i+n_1+n_2+\cdots+ n_{2l+1}$, then $Q'=(k,  n-i+n_1+n_2+\cdots n_{2l+1}-k)$. 
        Let the points of intersections of the paths $q_2$ and $q_2'$ with the line $X+Y = n-i+n_1+n_2+\cdots +n_{2l+1}$ be $R$ and $R'$, respectively. 
        Then 
		$\mathtt{IP}_{q_{_2},q_2'}(X+Y = n-i+n_1+n_2+\cdots +n_{2l+1} ~\&~ X=k)$ contains the points on line $RQ'$ except the point $R$.
		Now, using  $X(A_{l_2}) + Y(A_{l_2})= n-i$ and given value of $k$, we have
		\begin{align*}
			Y(R) - Y(Q') 
			&= (Y(A_{l_2})+n_2 +n_4 + \cdots+ n_{2l})- (n-i+n_1+n_2+\cdots+ n_{2l+1}-k)\\
			&= Y(A_{l_2}) - (n-i) -n_1 -n_3 - \cdots -n_{2l+1} +k  \\
			&= -X(A_{l_2}) -n_1 -n_3 - \cdots -n_{2l+1} + X(A_{l_1}) + n_1 + n_3 + \cdots + n_{2l-1} + n_{2l+1}'  \\
			&= X(A_{l_1}) -X(A_{l_2}) -n_{2l+1}  + n_{2l+1}'\\
			&= l_1 - l_2 -n_{2l+1}''.
		\end{align*}
        Thus,  
		$\mathtt{IP}_{q_{_2},q_2'}(X+Y = n-i+n_1+n_2+\cdots n_{2l+1}~ \&~ X=k) =$
        \begin{center}
        \small
            $ \{ (X(A_{l_2}) + n_1+n_3 +\cdots + n_{2l+1} + s_{2l+1}, Y(A_{l_2})+n_2 +n_4 + \cdots+ n_{2l} - s_{2l+1}) : 1 \leq s_{2l+1} \leq l_1- l_2 - n_{2l+1}''  \}$.
        \end{center}
        This proves $(ii)$.
        
        Now, we compute $\mathtt{IP}_{q_{_2},q_2'}(X+Y = n-i+n_1+n_2+\cdots n_{2l+2}~ \&~ X=k)$.
        Let $X+Y = n-i+n_1+n_2+\cdots n_{2l+2}$ intersect $X=k$ and $q_2$ at $Q''$ and $R''$, respectively. Then,
        $Y(R'')-Y(Q'') = Y(R)+ n_{2l+2} - (Y(Q') + n_{2l+2}) = Y(R)-Y(Q') = l_1-l_2-n_{2l+1}''$. 
        Thus, $\mathtt{IP}_{q_{q_2},q_2'}(X+Y = n-i+n_1+n_2+\cdots + n_{2l+2} ~\&~ X=k) = $
        \begin{center}
            \small
            $\{ (X(A_{l_2}) + n_1+n_3 +\cdots + n_{2l+1} + t_{2l+2}, Y(A_{l_2})+n_2 +n_4 + \cdots+  n_{2l+2} - t_{2l+2}) : 1 \leq t_{2l+2} \leq l_1- l_2 - n_{2l+1}''  \}$.
        \end{center}
        This proves $(iii)$.

        Let $s$ be an odd integer where $ 3 \leq s \leq 2r-2l-3$. We now find $\mathtt{IP}_{q_{_2},q_2'}(X+Y = n-i+  \sum_{i'=1}^{2l+s}n_{i'}  ~\&~ X=k) $.
            Let $X+Y = n-i+  \sum_{i'=1}^{2l+s}n_{i'} $ intersect $q_2$ and $X=k$ at $R^{(s)}$ and $Q^{(s)}$, respectively (see Figure \ref{fig1231er}).
            Then,
            \begin{align*}
                &Y(R^{(s)}) - Y(Q^{(s)}) \\
                &= (Y(A_{l_2})+n_2 +n_4 + \cdots+ n_{2l+s-1}) - (n-i+n_1+n_2+\cdots + n_{2l+s}-k)\\
                & = Y(A_{l_2}) - (n-i) - n_1 - n_3 - \cdots -n_{2l+s} +k\\
                & = Y(A_{l_2}) - (n-i) - n_1 - n_3 - \cdots -n_{2l+s} +
                 X(A_{l_1}) + n_{1} + n_{3} + \cdots + n_{2l-1} + n_{2l+1}'\\
                & = l_1 -l_2 -n_{2l+1}'' - n_{2l+3} - n_{2l+5} - \cdots - n_{2l+s}.
            \end{align*}

            Thus, 
            $\mathtt{IP}_{q_{_2},q_2'}(X+Y = n-i+n_1+n_2+\cdots + n_{2l+s} ~\&~ X=k) = $
            \begin{center}
                \small 
                 $\{ (X(A_{l_2}) +n_1+n_3 +\cdots + n_{2l+s} + s_{2l+s}, Y(A_{l_2})+n_2 +n_4 + \cdots + n_{2l+s-1} - s_{2l+s}) : 1 \leq s_{2l+s} \leq l_1- l_2 - n_{2l+1}'' - n_{2l+3} -n_{2l+5} - \cdots -n_{2l+s} \}$.
            \end{center}

            Now, we find  
			$\mathtt{IP}_{q_{_2},q_2'}(X+Y = n-i+  \sum_{i'=1}^{2l+s+1}n_{i'}  ~\&~ X=k) $.
             Let $X+Y = n-i+  \sum_{i'=1}^{2l+s+1}n_{i'} $ intersect $q_2$ and $X=k$ at $R^{(s+1)}$ and $Q^{(s+1)}$, respectively.

             Then $Y(R^{(s+1)}) - Y(Q^{(s+1)}) = Y(R^{(s)}) +n_{2l+s+1} - (Y(Q^{(s)}) + n_{2l+s+1}) = l_1 -l_2 -n_{2l+1}'' - n_{2l+3} - n_{2l+5} - \cdots - n_{2l+s}$.
             Thus, 
             $\mathtt{IP}_{q_{_2},q_2'}(X+Y = n-i+n_1+n_2+\cdots + n_{2l+s+1} ~\&~ X=k) = $
             \begin{center}
                 \small
                  $\{ (X(A_{l_2}) +n_1+n_3 +\cdots + n_{2l+s} + t_{2l+s+1}, Y(A_{l_2})+n_2 +n_4 + \cdots + n_{2l+s+1} - t_{2l+s+1}) : 1 \leq t_{2l+s+1} \leq l_1- l_2 - n_{2l+1}'' - n_{2l+3} -n_{2l+5} - \cdots -n_{2l+s} \}$.
             \end{center}
            This proves $(iv)$.
        
       \end{proof}
	
\begin{figure}[htbp]
			\centering
				\begin{tikzpicture}
					
					\draw[step=0.5cm,gray,very thin] (1.5,0) grid (8.5,8);

                    \draw[line width=.2mm]  (4.5,5.5) -- (4.5,5.5) node[above,right] {$R^{(s)}$};
                    \draw[line width=.2mm]  (6.5,3.5) -- (6.5,3.5) node[right] {$Q^{(s)}$};
                    \draw[line width=.2mm]  (4.5,7.5) -- (4.5,7.5) node[above] {$R^{(s+1)}$};
                    \draw[line width=.2mm]  (6.5,5.5) -- (6.5,5.5) node[right] {$Q^{(s+1)}$};
					
					\draw[dotted, line width=0.6mm] (1.5,4) -- (1.5,5.5);
					\draw[line width=.5mm]  (1.5,5.5) -- (2.5,5.5) node[below] {};
					\draw[dotted, line width=0.6mm]  (2.5,5.5) -- (3.5,5.5) node[above] {};
					\draw[line width=.5mm]  (3.5,5.5) -- (4.5,5.5) node[below] {}; 
					\draw[ line width=.6mm]  (4.5,5.4) -- (4.5,6.5) node[left] {$b^{n_{2l+s+1}}$};
					\draw[dotted, line width=.6mm]  (4.5,6.5) -- (4.5,7) node[left] {};
					\draw[line width=.6mm]  (4.5,7) -- (4.5,7.5) node[left] {};
					\draw[dotted, line width=.6mm]  (4.5,7.5) -- (5.5,7.5) node[below] {};
					\draw[<-, line width=.3mm]  (2,5.6) -- (2,6.5) node[above] {path $q_2$};
                    \draw[line width=0.3mm] (3,5.5) -- (3,5.5) node[midway, below] {$a^{n_{2l+s}}$};
					
					\draw[dotted, line width=0.6mm] (5,0.5) -- (5,2);
					\draw[line width=.5mm]  (5,2) -- (6,2) node[below] {};
					\draw[dotted, line width=0.6mm]  (6,2) -- (7,2) node[below] {};
					\draw[line width=.5mm]  (7,2) -- (8,2) node[midway, below] {};
					\draw[line width=.6mm]  (8,2) -- (8,3) node[right] {$b^{n_{2l+s+1}}$};
					\draw[dotted, line width=.6mm]  (8,3) -- (8,3.5) node[right] {};
					\draw[line width=.6mm]  (8,3.5) -- (8,4) node[below] {};
					\draw[dotted, line width=.6mm]  (8,4) -- (9,4) node[below] {};
                    \draw[line width=0.3mm] (5.5,2) -- (5.5,2) node[midway, above] {$a^{n_{2l+s}}$};
					\draw[<-, line width=0.3mm] (4.9,1.5) -- (4,1.5) node[left] {path $q_2'$};
					
					\draw[dotted, line width=.6mm]  (6.5,0) -- (6.5,1) node[right] {};
					\draw[line width=.5mm]  (6.5,1) -- (6.5,2.5) node[right] {};
					\draw[dotted, line width=.6mm]  (6.5,2.5) -- (6.5,3) node[right] {};
					\draw[line width=.6mm]  (6.5,3) -- (6.5,4.5) node[right] {};
					\draw[dotted, line width=.6mm]  (6.5,4.5) -- (6.5,5) node[right] {};
					\draw[line width=.5mm]  (6.5,5) -- (6.5,6.5) node[right] {$X=k$};
					\draw[dotted, line width=.6mm]  (6.5,6.5) -- (6.5,7.5) node[right] {};
					
					\draw[line width=.2mm]  (4.5,5.5) -- (5.5,4.5) node[below] {};
					\draw[dotted, line width=.6mm]  (5.5,4.5) -- (6,4) node[below] {};
					\draw[line width=.2mm]  (6,4) -- (8.5,1.5) node[below] {};
                    \draw[line width=.2mm]  (8.5,1.5) -- (8.5,1.5) node[right] {$X+Y = n-i+  \sum_{i'=1}^{2l+s}n_{i'} $};

					\draw[line width=.2mm]  (4.5,7.5) -- (5.5,6.5) node[below] {};
					\draw[dotted, line width=.6mm]  (5.5,6.5) -- (6,6) node[below] {};
					\draw[line width=.2mm]  (6,6) -- (8,4) node[below] {};
                    \draw[->,line width=.2mm]  (7.5,4.7) -- (8.5,4.7) node[right] {$X+Y = n-i+  \sum_{i'=1}^{2l+s+1}n_{i'} $};

                    \fill[black] (5,5) circle (1.5pt);
		            \draw (5,5) circle (0.15cm);
                    \fill[black] (5.5,4.5) circle (1.5pt);
		            \draw (5.5,4.5) circle (0.15cm);
                    \fill[black] (6,4) circle (1.5pt);
		            \draw (6,4) circle (0.15cm);
                    \fill[black] (6.5,3.5) circle (1.5pt);
		            \draw (6.5,3.5) circle (0.15cm);

                    \fill[black] (5,7) circle (1.5pt);
		            \draw (5,7) circle (0.15cm);
                    \fill[black] (5.5,6.5) circle (1.5pt);
		            \draw (5.5,6.5) circle (0.15cm);
                    \fill[black] (6,6) circle (1.5pt);
		            \draw (6,6) circle (0.15cm);
                    \fill[black] (6.5,5.5) circle (1.5pt);
		            \draw (6.5,5.5) circle (0.15cm);

				\end{tikzpicture}
				\caption{Elements of $\mathtt{IP}_{q_{_2},q_2'}(X+Y = n-i+  \sum_{i'=1}^{2l+s}n_{i'} ~ \&~ X=k)$ and $\mathtt{IP}_{q_{_2},q_2'}(X+Y = n-i+  \sum_{i'=1}^{2l+s+1}n_{i'} ~ \&~ X=k)$ for some $3 \leq s \leq 2r-2l-3$.}
				\label{fig1231er}
		\end{figure}
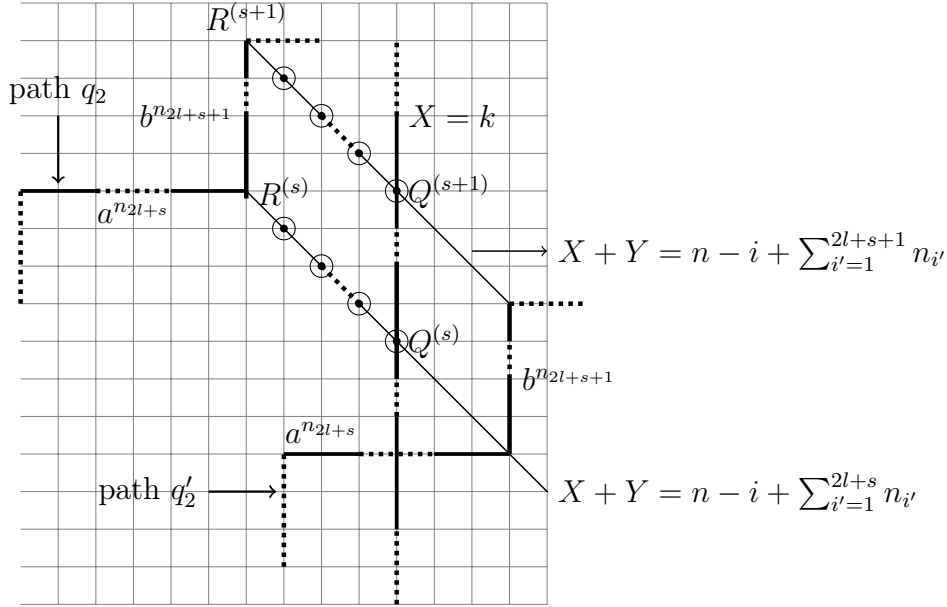

	\begin{proposition}\label{Prop1}
		For $|\gamma|\geq 3$, $l_1 - l_2 \geq 2$ and $Y(A_k)\neq Y(A_{l_2})$, the value of $ {}_{A_{l_2}}{{N}}_{A_{l_1}}^{A_k}$  is 
		\begin{center}
			\tiny
			$\displaystyle\sum_{\substack{E}} \sum_{\substack{D}} \sum_{\substack{C}} {n_1-1 \choose l_2 -l_1 +n_1 + s_1} {n_2 \choose t_2-s_1} {n_{2r-1}-1 \choose n_{2r-1} - t_{2r-2}} \displaystyle \prod_{r''=2}^{r-1} { {n_{2r''-1} \choose n_{2r''-1} + s_{2r''-1} - t_{2r''-2}} {n_{2r''} \choose t_{2r''} - s_{2r''-1}} }$,
		\end{center}
		in which $E = \{r \in \mathds{N}: 2 \leq r \leq \floor{\frac{i+j-n+1}{2}} \}$; $D$ consists of tuples $(n_1, n_2, \cdots, n_{2r-1})$ satisfying $\sum_{i'=1}^{2r-1}n_{i'} =|\gamma|$,  where $\sum_{j'=1}^{r}n_{2j'-1}=k-l_2$ and $\sum_{j'=1}^{r-1}n_{2j'} =  |\gamma| -k + l_2$; and
		$C$ is the collection of the following conditions: 
		\begin{enumerate}[label=\rm (\roman*)]
			\item $1 \leq s_1, t_2, s_3, t_4, \cdots, s_{2l-1}, t_{2l} \leq l_1 - l_2 - 1 $

			\item $1 \leq s_{2l+1}, t_{2l+2} \leq l_1 - l_2 - n_{2l+1}''$

			\item $1 \leq s_{2l+s}, t_{2l+s+1} \leq l_1 - l_2 - n_{2l+1}'' - n_{2l+3} - \cdots - n_{2l+s} \text{ for } 3 \leq s\leq 2r-2l-3$ 
		\end{enumerate}
		where $0 \leq l \leq r-1$, $n_{2l+1}=n_{2l+1}' + n_{2l+1}''$ with $n_{2l+1}'\geq 0 $ and $n_{2l+1}'' > 0$, and $l_1-l_2 = n_{2l+1}'' + n_{2l+3} +\cdots + n_{2r-1}$.
		
	\end{proposition}

	\begin{proof}

		
		For $r\in \mathds{N}$ and integers $n_1, n_2, \cdots, n_{2r-1} \geq 1$, let  $P = (n_1, n_2,  \cdots, n_{2r-1})$ be a $(2r-1)$-partition of $i+j-n$ such that $w_{A_{l_2}}^{A_k}(q_2) =a^{n_{1}} b^{n_{2}} a^{n_{3}}\cdots  b^{n_{2r-2}} a^{n_{2r-1}}$ for some lattice path $q_2$ from $A_{l_2}$ to $A_k$.
        Since $Y(A_k)\neq Y(A_{l_2})$, $r \geq 2$.
        Let $q_2'$ be a lattice path from $A_{l_1}$ corresponding to the word $w_{A_{l_2}}^{A_k}(q_2)$ (see Fig. \ref{fig3n1}).
		Using $2D$ geometry, we now compute
		$\mathtt{IP}_{q_{_2},q_{2}'}(X+Y = n-i+ \sum_{i'=1}^{s'} n_{i'}~ \&~ X = k)$ for each $1 \leq s' \leq 2r-2$ (see Lemma \ref{internal points}).
		In short, we denote this set as  $\mathtt{IP}_{q_{_2},q_{2}'}^{P}(s')$ where $1 \leq s' \leq 2r-2$. Furthermore, by geometric arguments, we observe that $\mathtt{IP}_{q_{_2},q_{2}'}^{P} (s')$ is nonempty for every $1 \leq s' \leq 2r-2$.


		
		Now we calculate all possible lattice paths $p_2$ that begin at $A_{l_1}$, proceed to $E_1$, then move sequentially through the points of  $\mathtt{IP}_{q_{_2},q_{2}'}^{P} (1)$,  $\mathtt{IP}_{q_{_2},q_{2}'}^{P} (2)$, $\ldots$,  $\mathtt{IP}_{q_{_2},q_{2}'}^{P} (2r-2)$, followed by $E_2$, and finally terminate at $A_k$ (see Fig. \ref{fig3n1}). 
		From Lemma \ref{nbve}, these paths $p_2$ intersect the paths $q_2'$ and $q_2$ only at $A_{l_1}$ and $A_k$, respectively.
		From these paths $p_2$, we now construct the set of all possible triplets $(q_2, p_2, q_2')$. Since this set depends on $q_2$, and $q_2$ depends on the partition $P$,
		we denote it by $G(P)$.

		Since $w_{A_{l_2}}^{A_k}(q_2) =a^{n_{1}} b^{n_{2}} a^{n_{3}}\cdots  b^{n_{2r-2}} a^{n_{2r-1}}$, we have $n_1 + n_2 + \cdots+ n_{2r-1} = i+j-n$ with  $ n_1 + n_3 + \cdots + n_{2r-1}=k-l_2$ and $ n_2 + n_4 + \cdots + n_{2r-2} = (i+j-n) -(k - l_2)$.
		Then there are $k-l_2-1 \choose r-1$  $i+j-n -k+l_2 -1\choose r-2$ distinct $(2r-1)$-partitions of the integer $i+j-n$
		where sum of odd-positioned terms is $k-l_2$ and sum of even-positioned terms is $i+j-n-k+l_2$. 
		Let $\mathscr{P}$ be the set of all such $(2r-1)$-partitions of $i+j-n$.
		Then for each $P = (n_1, n_2, \cdots, n_{2r-1}) \in \mathscr{P}$, we consider a path $q_2$ from $A_{l_2}$ to $A_k$ such that $w_{A_{l_2}}^{A_k}(q_2) = a^{n_1} b^{n_2} \cdots a^{n_{2r-1}}$ and compute $G(P)$.
		Consider $H(2r-1) = \bigcup_{P \in \mathscr{P}}   G(P)$. Now, $ 3 \leq 2r-1 \leq i+j-n$, i.e., $2 \leq r \leq \floor{\frac{i+j-n+1}{2}}$. Let $T= \bigcup_{r = 2}^{\floor{\frac{i+j-n+1}{2}}} H(2r-1)$.
		We now show that  $ {}_{A_{l_2}}{{\mathcal{N}}}_{A_{l_1}}^{A_k} =T$. 
		
		Let $(q_0, p_0, q_0') \in {}_{A_{l_2}}{\mathcal{N}}_{A_{l_1}}^{A_k}$. 
		Then $q_0$ is a path from $A_{l_2}$ to $A_k$,  $p_0$ is a path from $A_{l_1}$ to $A_k$, $q_0'$ is a path from $A_{l_1}$ corresponding to the word $ w_{A_{l_2}}^{A_k}(q_0)$, $q_0 \cap p_0 = \{ A_k\}$, and $p_0 \cap q_0' = \{A_{l_1}\}$.
		Now, $|w_{A_{l_2}}^{A_k}(q_0)|=|w_{A_{l_1}}^{A_k}(p_0)|=i+j-n \geq 3$. Then by Remark \ref{imprem231}, there exists a partition $P_0 = (n_1', n_2', \cdots, n_{2r'-1}')$ of $i+j-n$
		such that $w_{A_{l_2}}^{A_k}(q_0)  = a^{n_1'} b^{n_2'} a^{n_3'} \cdots a^{n_{2r'-1}'}$  
		and $w_{A_{l_1}}^{A_k}(p_0)$ starts and ends with $b$ (since $Y(A_k)\neq Y(A_{l_2})$, $r' \geq 2$). 
		Thus, the path $p_0$ starts at $A_{l_1} = (l_1, n-i-l_1)$, then it moves to $E_1$. Since $q_0 \cap p_0 = \{ A_k\}$, and $p_0 \cap q_0' = \{A_{l_1}\}$, the path $p_0$ subsequently passes through some points of $\mathtt{IP}_{q_{_0},q_0'}^{P_0}(1)$, $\mathtt{IP}_{q_{_0},q_0'}^{P_0}(2), \cdots$,
		$\mathtt{IP}_{q_{_0},q_0'}^{P_0}(2r'-2)$ and then $E_2$ and finally reaches $(k, j-k)$.
		Since $\mathtt{IP}_{q_{_0},q_0'}^{P_0}(1)$, $\mathtt{IP}_{q_{_0},q_0'}^{P_0}(2), \cdots$,
		$\mathtt{IP}_{q_{_0},q_0'}^{P_0}(2r'-2)$ are  nonempty, there always exists a route for $p_0$.
		Thus, $(q_0, p_0, q_0')$ is an element of $G(P_0)$. This implies $(q_0, p_0, q_0') \in T$. Thus $ {}_{A_{l_2}}{\mathcal{N}}_{A_{l_1}}^{A_k} \subseteq T$.

		Let $(q_2, p_2, q_2') \in T$. Then for some $3 \leq 2r''-1 \leq i+j-n$, $(q_2, p_2, q_2') \in H(2r''-1)$. Then for some $(2r''-1)$-partition ${P}''=(n_1'', n_2'', \cdots, n_{2r''-1}'')$ of $i+j-n$ where $n_1'' + n_3'' + \cdots + n_{2r''-1}''=k-l_2$ and $n_2'' + n_4'' + \cdots + n_{2r''-2}'' = i+j-n -k + l_2$, we have $(q_2, p_2, q_2') \in G({P}'')$. 
		So, $q_2$ is a path from $A_{l_2}$ to $A_k$, $q_2'$ is a path from $A_{l_1}$ corresponding to word $w_{A_{l_2}}^{A_k}(q_2)$, $p_2$ is a path from $A_{l_1}$ to $A_k$ such that $q_2 \cap p_2=\{A_k\}$ and $p_2 \cap q_2'=\{A_{l_1} \}$. Thus $(q_2, p_2, q_2') \in {}_{A_{l_2}}{\mathcal{N}}_{A_{l_1}}^{A_k}$. Hence,  $ T \subseteq {}_{A_{l_2}}{\mathcal{N}}_{A_{l_1}}^{A_k}$ so that $ T = {}_{A_{l_2}}{\mathcal{N}}_{A_{l_1}}^{A_k}$.

		Now we calculate the cardinality of $T$. As a first step, we compute the cardinality of $G({P})$. This is straightforward since the coordinates of the internal points are known by Lemma \ref{internal points}, and the formula for the number of lattice paths between two points is available.
        For $1 \leq i' \leq r-1$, let $B_{s_{2i'-1}}$ and $B_{t_{2i'}}$ denote arbitrary points of $\mathtt{IP}_{q_{_2},q_{2}'}^{P} (2i'-1)$ and $\mathtt{IP}_{q_{_2},q_{2}'}^{P} (2i')$, respectively. Then from Lemma \ref{internal points}, 
        
        $B_{s_{2i'-1}} = (X(A_{l_2})+ \sum_{j'=1}^{i'}n_{2j'-1} + s_{2i'-1}, Y(A_{l_2})+ \sum_{j'=1}^{i'-1}n_{2j'} - s_{2i'-1} )$, and
        
        $B_{t_{2i'}} = (X(A_{l_2}) + \sum_{j'=1}^{i'}n_{2j'-1}  + t_{2i'}, Y(A_{l_2})+ \sum_{j'=1}^{i'}n_{2j'} - t_{2i'})$\\
        with some appropiate values of $s_{2i'-1}$ and $t_{2i'}$.

         Now,  $N_{A_{l_1}}^{E_1}=1$,  $N_{E_1}^{B_{s_1}} = {n_1-1 \choose l_2 -l_1 +n_1 + s_1}$ and 
         $N_{B_{s_1}}^{B_{t_2}}   ={n_2 \choose t_2-s_1}$.
         Also, for any $2 \leq j' \leq r-1$, 
         $$N_{B_{t_{2j'-2}}}^{B_{s_{2j'-1}}}= {n_{2j'-1} \choose n_{2j'-1} + s_{2j'-1} - t_{2j'-2}} \text{ and }  N_{B_{s_{2j'-1}}}^{B_{t_{2j'}}} ={n_{2j'} \choose t_{2j'} - s_{2j'-1}}.$$

		 Furthermore,  $N_{B_{t_{2r-2}}}^{E_2}  = {n_{2r-1}-1 \choose n_{2r-1} - t_{2r-2}}$, and $N_{E_2}^{A_k}=1$.
        Thus
		$|G(P)| = $
		\begin{center}\small
			$\displaystyle\sum_{\substack{C}} {n_1-1 \choose l_2 -l_1 +n_1 + s_1} {n_2 \choose t_2-s_1} {n_{2r-1}-1 \choose n_{2r-1} - t_{2r-2}} \displaystyle \prod_{j'=2}^{r-1} { {n_{2j'-1} \choose n_{2j'-1} + s_{2j'-1} - t_{2j'-2}} {n_{2j'} \choose t_{2j'} - s_{2j'-1}} }. $
		\end{center}
		Here $C$ is the collection of the following conditions that arise in the determination of the coordinates of the internal points: 
		$1 \leq s_1, t_2, s_3, t_4, \cdots, s_{2l-1}, t_{2l} \leq l_1 - l_2 - 1$, ~
		$1 \leq s_{2l+1}, t_{2l+2} \leq l_1 - l_2 - n_{2l+1}''$,
		$1 \leq s_{2l+s}, t_{2l+s+1} \leq l_1 - l_2 - n_{2l+1}'' - n_{2l+3} - \cdots - n_{2l+s}$,
		$\text{ for odd integer } s ~(3 \leq s \leq 2r-2l-3)$ where $n_{2l+1}=n_{2l+1}' + n_{2l+1}''$ for some $0 \leq l \leq r-1$ with $n_{2l+1}' \geq 0$, $n_{2l+1}'' > 0$ and $l_1-l_2 = n_{2l+1}'' + n_{2l+3} +\cdots + n_{2r-1}$.

		Then we vary ${P}$ over all partitions in $\mathscr{P}$ to obtain the cardinality of $H(2r-1)$. This gives
		$ |H(2r-1)| =  \sum_{P \in \mathscr{P}} 
		|G(P)|$.
		Finally, we vary the value of $r$ from $2$ to $\floor{\frac{i+j-n+1}{2}}$ and compute $|T|$. This gives
		$|T|= \displaystyle \sum_{r = 2}^{\floor{\frac{i+j-n+1}{2}}} |H(2r-1)|.$
	\end{proof}

	\begin{proposition}\label{newcase}
		For $|\gamma|\geq 3$, $l_1 - l_2 \geq 2$ and $Y(A_k) = Y(A_{l_2})$, the value of $ {}_{A_{l_2}}{{N}}_{A_{l_1}}^{A_k}$  is ${|\gamma| -2 \choose k-l_1}$.		
	\end{proposition}
    \begin{proof}
        Since $Y(A_k) = Y(A_{l_2})$, $w_{A_{l_2}}^{A_k}=a^{|\gamma|}$. Then $k = l_2+|\gamma|$ and there exists exactly one lattice path, say $q_2$, from $A_{l_2}$ to $A_k$. 
        Let $q_2'$ be a lattice path from $A_{l_1}$ corresponding to $w_{A_{l_2}}^{A_k}$. 
        Now, we calculate all possible lattice paths $p_2$ from $A_{l_1}$ to $A_{k}$ such that $q_2 \cap p_2 = \{A_k\}$ and $q_2' \cap p_2 = \{ A_{l_1} \}$. By Remark \ref{imprem231}, we know that $p_2$ starts and end with $b$. So, $p_2$ passes through $E_1$ and $E_2$. Since  $Y(A_k) = Y(A_{l_2})$, the value of $ {}_{A_{l_2}}{{N}}_{A_{l_1}}^{A_k}$  is given by the total number of possible lattice path between $E_1$ and $E_2$.
        Therefore, 
         $ {}_{A_{l_2}}{{N}}_{A_{l_1}}^{A_k} = { Y(E_2) - Y(E_1)+ X(E_2)-X(E_1) \choose X(E_2)-X(E_1)} = { |\gamma|-2 \choose k-l_1  }$ (see Figure \ref{equal lk}). 
         \end{proof}

                    \begin{figure}[htbp]
					\centering
					\begin{tikzpicture}

                    \draw[dotted, ->, line width=0.15mm] (-2,-0.5) -- (7,-0.5) node[midway, above] {};
				    \draw[line width=.1mm] (7,0) -- (7,0) node[midway,below] {$X$};
				    \draw[dotted, ->, line width=0.15mm] (-1,-1) -- (-1,5) node[midway, above] {};
				    \draw[line width=.1mm] (-1,5) -- (-1,5) node[midway, above] {$Y$};
						
						\draw[line width=.5mm]  (-0.5,2.5) -- (2,0) node[below] {};
						\draw[line width=.5mm]  (-0.5,2.5) -- (-0.5,2.5) node[below] {$A_{l_2}$};
						\fill[red] (-0.5,2.5) circle (2pt);
						\draw[line width=.5mm]  (1.5,0.5) -- (1.5,0.5) node[below] {$A_{l_1}$};
						\fill[red] (1.5,0.5) circle (2pt);
						\draw[line width=.5mm]  (2,0) -- (2,0) node[below] {$L_1$};
						
						\draw[line width=.5mm]  (2.5,3.5) -- (6,0) node[below] {};
                        \fill[red] (3.5,2.5) circle (2pt);
						\draw[line width=.5mm]  (3.5,2.5) -- (3.5,2.5) node[above] {$A_{k}$};
						\draw[line width=.5mm]  (6,0) -- (6,0) node[below] {$L_2$};
						
						\draw[line width=.5mm]  (-0.5,2.5) -- (3.5,2.5) node[midway, above] {$q_2$};
						\draw[line width=.5mm]  (1.5,0.5) -- (5.5,0.5) node[midway, below] {$q_2'$};
						\fill[red] (5.5,0.5) circle (2pt);
						\draw[line width=.5mm]  (5.5,0.5) -- (5.5,0.5) node[above] {$A_{k'}$};

                        \draw[line width=.5mm]  (1.5,0.5) -- (1.5,1) node[midway, below] {};
                        \fill[red] (1.5,1) circle (2pt);
						\draw[line width=.5mm]  (1.5,1) -- (1.5,1) node[above] {$E_1$};

                        \draw[line width=.5mm]  (3.5,2.5) -- (3.5,2) node[midway, below] {};
                        \fill[red] (3.5,2) circle (2pt);
						\draw[line width=.5mm]  (3.5,2) -- (3.5,2) node[below] {$E_2$};

                        \draw[dotted, line width=0.5mm] (1.5,1) .. controls (2.5,1) .. (3.5,2) node[midway, above] {$p_2$};	
    
					\end{tikzpicture}
					\caption{Illustration of proof of Proposition \ref{newcase}}
					\label{equal lk}
			\end{figure}
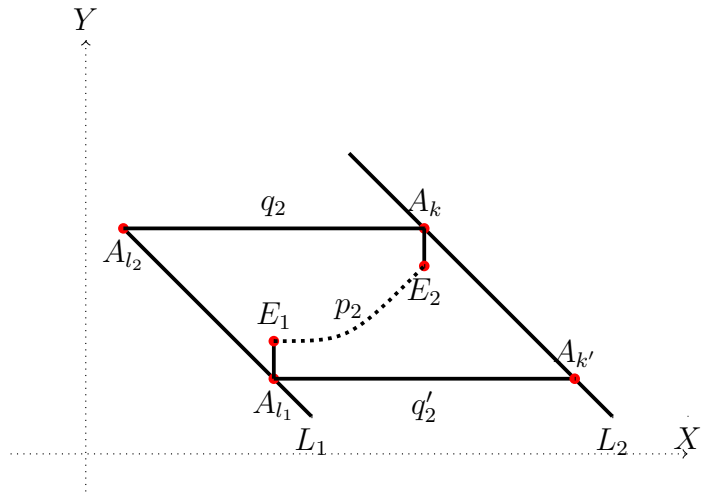

	\begin{example}
		We illustrate propositions \ref{Prop1} and \ref{newcase} with the help of following examples:
		\begin{enumerate}
			\item 
			Let $A_{l_1} = (4, 1)$, $A_{l_2} = (1, 4)$ and $A_k = (5, 5)$.    
			Let $q_2$ be a lattice path from $A_{l_2}$ to $A_k$ which begin and end with $a$. Consider  $w_{A_{l_2}}^{A_k}(q_2) =a^{n_{1}} b^{n_{2}} a^{n_{3}}\cdots   a^{n_{2r-1}}$ for some $r \geq 1$. Since $Y(A_k) \neq Y(A_{l_2})$, $r \geq 2$. 
            Also as $Y(A_k)-Y(A_{l_2})=1>0$, $|w_{A_{l_2}}^{A_k}(q_2)|_b=1$, i.e., $r=2$, i.e.,
			$w_{A_{l_2}}^{A_k}(q_2)$ must be in the form $a^{n_1} b^{n_2} a^{n_3}$ where $n_1+n_3=4$ and $n_2=1$.
			Now, the number of positive solutions of the equation $n_1+n_3=4$ is ${4-1 \choose 2-1} =3$. 
			So, there are $3$ partitions of $5$: let ${P}_1=(1, 1, 3)$, ${P}_2=(3, 1, 1)$ and ${P}_3=(2, 1, 2)$ be the partitions of $5$. 
			Then $w_{A_{l_2}}^{A_k}(q_2)$ is one of the following: 
			$aba^3$, $a^3 b a$ and $a^2 b a^2$. 
			\begin{itemize}
				\item When $w_{A_{l_2}}^{A_k}(q_2) = aba^3$. Then by Equation \ref{foundl}, we have $l=r-1=1$, $n_3'=0$ and $n_3''=3$. Also, $1 \leq s_1 \leq 2$,  $1 \leq t_2 \leq 2$. 
				Thus by Proposition \ref{Prop1}, $|G(P_1)| = 2$.

				\item When $w_{A_{l_2}}^{A_k}(q_2) = a^3ba$. Then by Equation \ref{foundl},
                we have $l=r-2=0$, $n_1'=1$ and $n_1''=2$. Also, $s_1 = 1$,  $t_2 = 1$. 
				Thus by Proposition \ref{Prop1}, $|G(P_2)| = 2$.

				\item When $w_{A_{l_2}}^{A_k}(q_2) = a^2ba^2$. Then, as $ l_1-l_2 = n_{2l+1}'' + n_{2l+3} +\cdots + n_{2r-1}$, we have $l=r-2=0$, $n_1'=1$ and $n_1''=1$. Also, $1 \leq s_1 \leq 2$,  $1 \leq t_2 \leq 2$. 
				Thus by Proposition \ref{Prop1}, $|G(P_3)| = 3$. 
			\end{itemize}
			Therefore, $h(3)=7$. This implies ${}_{A_{l_2}}{{N}}_{A_{l_1}}^{A_k}=7$.

            \item 
			Let $A_{l_1} = (5, 0)$, $A_{l_2} = (1, 4)$ and $A_k = (9, 4)$.    
			Let $q_2$ be a lattice path from $A_{l_2}$ to $A_k$. Since $Y(A_k) = Y(A_{l_2})$,  $q_2 = a^{8}$. Then $E_1 = (5,1)$ and $E_2 = (9,3)$. Then by Proposition \ref{newcase}, ${}_{A_{l_2}}{{N}}_{A_{l_1}}^{A_k}= 15$.
			
		\end{enumerate}
	\end{example}


    From propositions \ref{niwe}, \ref{Prop1} and \ref{newcase}, we have the following:
	\begin{theorem} \label{gammagreat5}
		For $|\gamma| = i+j-n \geq 3$, the value of $\mathcal{M}^{\textup{o}}(n)$  is 
		\begin{center}
			\small
			$ 2 \displaystyle \sum_j \sum_i \sum_{l_2=0}^{n-i-2} \sum_{l_1=l_2+2}^{n-i} \sum_{k=l_1}^{i+j-n+l_2}  \sum_{m=k'}^{n-j+k}   \left(N_{O}^{A_{l_1}, A_{l_2}}\right)   \left({}_{A_{l_2}}{{N}}_{A_{l_1}}^{A_k}\right)   \left(N_{A_k, A_{k'}}^{A_{m}}\right)$
		\end{center}
		where $ N_{O}^{A_{l_1}, A_{l_2}}$, ${}_{A_{l_2}}{{N}}_{A_{l_1}}^{A_k}$ and  $N_{A_k, A_{k'}}^{A_{m}}$ are given by propositions \ref{niwe}, \ref{Prop1} and \ref{newcase}.
	\end{theorem}

    We illustrate Theorem \ref{gammagreat5} with the help of following examples:
    	\begin{example}\label{gamma example 45}
		For $3 \leq|\gamma| \leq 4$, the values of $\mathcal{M}^{\textup{o}}(n)$  are as per the following:
		\begin{itemize}
         \small
					
			\item When $|\gamma|=3$, then $\mathcal{M}^{\textup{o}}(n)=$ 
			$$2 \sum_{i=5}^{n-2} \sum_{l_2=0}^{n-i-2} \sum_{l_1=2+l_2 }^{3+l_2} \sum_{k=l_1}^{l_2+3}  \sum_{m=k+l_1-l_2}^{n-j+k}   
			\frac{(l_1 - l_2)^2}{(n-i)(n-j)} {n-i \choose l_1} {n-i \choose l_2} {n-j \choose m-k} {n-j \choose m-k-l_1+l_2} {}_{A_{l_2}}{{N}}_{A_{l_1}}^{A_k}$$
			
			where 
			\[
			{}_{A_{l_2}}{{N}}_{A_{l_1}}^{A_k}  = 
			\begin{cases}
				1 & \text{if } l_1=2+l_2 ~ \& ~ k = l_1 \text{ or } l_1+1, \\
				1 & \text{if } l_1=3+l_2 ~\&~  k=l_1,\\
				0 & \text{ otherwise.} 
			\end{cases}
			\]
			
			\item 
			When $|\gamma|=4$, then $\mathcal{M}^{\textup{o}}(n)=$ 
			$$2 \sum_{i=6}^{n-2} \sum_{l_2=0}^{n-i-2} \sum_{l_1=2+l_2 }^{4+l_2} \sum_{k=l_1}^{l_2+4}  \sum_{m=k+l_1-l_2}^{n-j+k}  \frac{(l_1 - l_2)^2}{(n-i)(n-j)} {n-i \choose l_1} {n-i \choose l_2}  {n-j \choose m-k} {n-j \choose m-k-l_1+l_2} {}_{A_{l_2}}{{N}}_{A_{l_1}}^{A_k}  $$
			where 
			\[
			{}_{A_{l_2}}{{N}}_{A_{l_1}}^{A_k}  = 
			\begin{cases}
				1 & \text{if } (l_1=2+l_2 ~\& ~  k = l_1 \text{ or } l_1+2) \text{ or } (l_1=4+l_2 ~\& ~ k=l_1), \\
				2 & \text{if } (l_1=2+l_2 ~\& ~ k=l_1+1) \text{ or } (l_1=3+l_2 ~\& ~ k=l_1 \text{ or } l_1+1 ),\\
				0 & \text{otherwise.} 
			\end{cases}
			\]
		\end{itemize}

	\textbf{Explanation: }	Since we can move only horizontally or vertically along a lattice path,
		$\max\{l_1,l_2\} \leq k \leq \min\{l_1, l_2\} + |\gamma|$ and  
		$ \max\{k, k' \} \leq m \leq \min\{k, k'\} +n-j$ (see Fig. \ref{yytr}).
		We now compute the value of ${}_{A_{l_2}}{{N}}_{A_{l_1}}^{A_k}$ for  $3 \leq|\gamma| \leq 4$.
		
		\begin{itemize}
			
	\item   Let $|{\gamma}| =3$, i.e.,  $|\gamma'|=i+j-n =3$. By Lemma \ref{prop120}, $||{\alpha}|_c-|{\alpha'}|_c| \geq 2 ~\forall c \in \{a, b\}$.
			Let $|{\alpha}|_a-|{\alpha'}|_a \geq 2$. Then $l_1-l_2\geq 2$. Now, $l_1 \leq k \leq l_2 + 3$. This implies $2 \leq l_1 - l_2 \leq 3$. We now have the following cases:
			\begin{itemize}
				\item Let $l_1 - l_2 = 2$: Then $l_2 + 2 \leq k \leq l_2 + 3$. 
				
				When $k=l_2+2$, then $Y(A_k)=j-k=n-i+3-l_2-2 = n-i-l_2 +1$ and $A_k= (l_2+2, n-i-l_2+1)$ (see Fig. \ref{k=l_2+2_l_1=l_2+2}). Now, by Lemma \ref{lem1}, $\gamma=b x b$ and ${\gamma'}^R=aya$ for some $x, y \in \Sigma^+$. Since $X(A_{l_1})=X(A_k)$, $\gamma = bbb$. Since $Y(A_k)-Y(A_{l_2})=1$ and ${\gamma'}^R=aya$, we have ${\gamma'}R=aba$ (see Fig. \ref{k=l_2+2_l_1=l_2+2}).
				Thus ${}_{A_{l_2}}{{N}}_{A_{l_1}}^{A_k} =1$.
				
				 When $k=l_2+3$, then $Y(A_k)=Y(A_{l_2})$ (see Fig. \ref{k=l_2+3_l_1=l_2+2}). Then by Proposition \ref{newcase} we have ${}_{A_{l_2}}{{N}}_{A_{l_1}}^{A_k} =1$.

				\item Let $l_1 - l_2 =3$: Then $l_2 + 3 \leq k \leq l_2 + 3$, i.e., $ k = l_2 + 3$. Then $Y(A_k) = j-k = n-i+3 - l_2-3=n-i-l_2 = Y(A_{l_2})$ and $A_k=(l_2+3, n-i-l_2) = (l_1, n-i-l_2)$ (see  Fig. \ref{k=l_2+3_l_1=l_2+3}).
				Then from Fig. \ref{k=l_2+3_l_1=l_2+3}, ${}_{A_{l_2}}{{N}}_{A_{l_1}}^{A_k} =1$.
			\end{itemize}
			Thus the value of $\mathcal{M}^{\textup{o}}(n)$ for $|\gamma|=3 $ is
			\begin{center}
				\tiny
				$2 \displaystyle \sum_{i=5}^{n-2} \sum_{l_2=0}^{n-i-2} \sum_{l_1=2+l_2 }^{3+l_2} \sum_{k=l_1}^{l_2+3}  \sum_{m=k+l_1-l_2}^{n-j+k}   
				\frac{(l_1 - l_2)^2}{(n-i)(n-j)} {n-i \choose l_1} {n-i \choose l_2}   {n-j \choose m-k} {n-j \choose m-k-l_1+l_2} {}_{A_{l_2}}{{N}}_{A_{l_1}}^{A_k}$
			\end{center}
			where 
			\[
			{}_{A_{l_2}}{{N}}_{A_{l_1}}^{A_k}  = 
			\begin{cases}
				1 & \text{if } l_1=2+l_2, ~ l_2+2\leq k \leq l_2+3 \\
				1 & \text{if } l_1=3+l_2,  k=l_2+3\\
				0 & \text{ otherwise} 
			\end{cases}
			\]
			Here note that $i \geq 5$ as $|\gamma|=3$ and $|\beta|\geq 2$. 
			
			\begin{figure}[htbp]
				\centering
				\begin{minipage}{0.3\textwidth}
					\centering
					\includegraphics[scale=0.5]{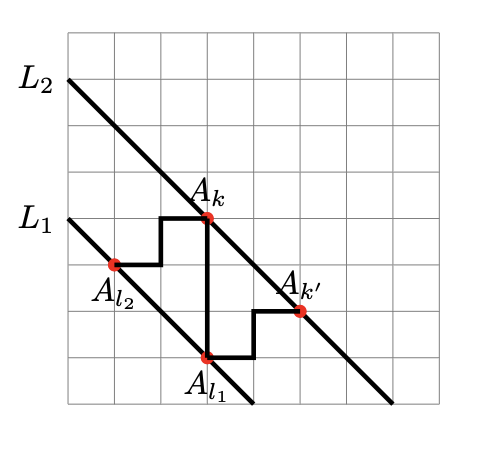}
					\caption{Illustration to the proof of Proposition \ref{Inticase} when $|\gamma|=3$, $k=l_2+2$,  $l_1-l_2=2$} 
					\label{k=l_2+2_l_1=l_2+2}
				\end{minipage}
				\begin{minipage}{0.3\textwidth}
					\centering
					\includegraphics[scale=0.5]{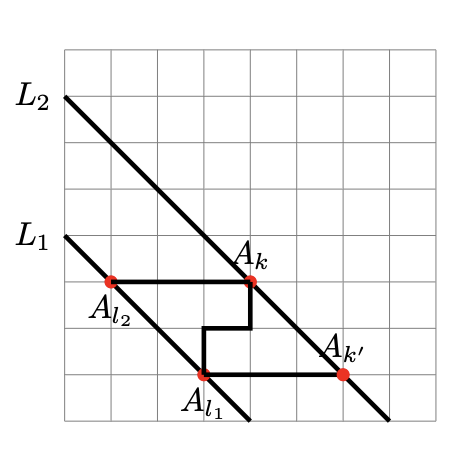}
					\caption{Illustration to the proof of Proposition \ref{Inticase} when $|\gamma|=3$, $k=l_2+3$,  $l_1-l_2=2$}
					\label{k=l_2+3_l_1=l_2+2}
				\end{minipage}
				\begin{minipage}{0.3\textwidth}
					\centering
					\includegraphics[scale=0.5]{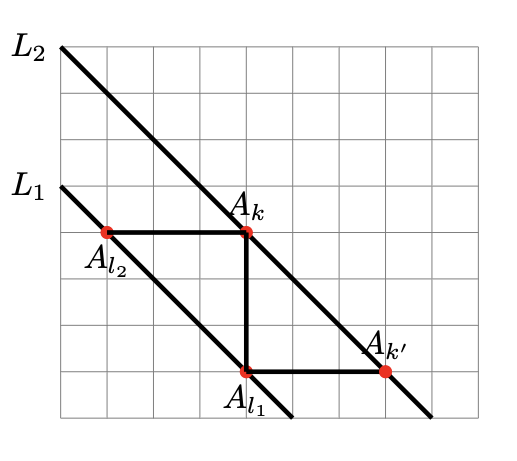}
					\caption{Illustration to the proof of Proposition \ref{Inticase} when $|\gamma|=3$, $k=l_2+3$,  $l_1-l_2=3$}
					\label{k=l_2+3_l_1=l_2+3}
				\end{minipage}
			\end{figure}

			\item  Let  $|{\gamma}| =4$, i.e., $|{\gamma'}|=i+j-n =4$.  By Lemma \ref{prop120}, $||{\alpha}|_c-|{\alpha'}|_c| \geq 2 ~\forall c \in \{a, b\}$. Let $|{\alpha}|_a-|{\alpha'}|_a \geq 2$, i.e., $l_1-l_2\geq 2$. Now, $l_1 \leq k \leq l_2 + 4$. This implies $2 \leq l_1 - l_2 \leq 4$. We now have the following cases:
			\begin{itemize}
				\item Let $l_1 - l_2 = 2$. Then  $l_2 +2 \leq k \leq l_2 + 4$.
				
				When $k=l_2+4$, then $Y(A_k)=j-k = n-i+4 -l_2-4 = n-i-l_2=Y(A_{l_2})$ (see Fig. \ref{0_k=l_2+4_l_1=l_2+2}). 
                Then by Proposition \ref{newcase}, ${}_{A_{l_2}}{{N}}_{A_{l_1}}^{A_k}=1$.
                
				
				When $k=l_2+3$, then  $Y(A_k)=j-k = n-i +4 -l_2-3 = n-i-l_2+1$ and $X(A_k)= k = l_2+3 =l_1-2+3=l_1+1$ (see figures \ref{0_k=l_2+3_l_1=l_2+2} and \ref{0_k=l_2+3_l_1=l_2+2_1}).  Now, by Lemma \ref{lem1}, $\gamma=b x b$ and ${\gamma'}^R=aya$ for some $x, y \in \Sigma^+$. Since $X(A_k)-X(A_{l_1}) =1$, either $\gamma=b ba b$ or $\gamma=b ab b$. When $\gamma=b ba b$, then ${\gamma'}^R=a ba a$ (see Fig. \ref{0_k=l_2+3_l_1=l_2+2}). When  $\gamma=b ab b$, then ${\gamma'}^R=a ab a$ (see Fig. \ref{0_k=l_2+3_l_1=l_2+2_1}). Thus,   
				${}_{A_{l_2}}{{N}}_{A_{l_1}}^{A_k}=2$.

				When $k=l_2+2$, then similar to the above, ${}_{A_{l_2}}{{N}}_{A_{l_1}}^{A_k}=1$ (see Fig. \ref{0_k=l_2+2_l_1=l_2+2}).

				\item Let $l_1 - l_2 = 3$. Then  $l_2 +3 \leq k \leq l_2 + 4$.
				
				When $k=l_2+3$, then similar to the above ${}_{A_{l_2}}{{N}}_{A_{l_1}}^{A_k}=2$ (see figures \ref{0_k=l_2+3_l_1=l_2+3} and \ref{0_k=l_2+3_l_1=l_2+3_1}).
				
				When $k=l_2+4$, then similar to the above ${}_{A_{l_2}}{{N}}_{A_{l_1}}^{A_k}=2$ (see figures \ref{0_k=l_2+4_l_1=l_2+3} and \ref{0_k=l_2+4_l_1=l_2+3_1}).

				\item Let $l_1 - l_2 = 4$.  Then $ k = l_2 + 4$.
				
				Then similar to the above, ${}_{A_{l_2}}{{N}}_{A_{l_1}}^{A_k}=1$ (see Fig. \ref{0_k=l_2+4_l_1=l_2+4}).
				
			\end{itemize}
			
			Thus the value of $\mathcal{M}^{\textup{o}}(n)$ for $|\gamma|=4$ is
			\begin{center}
				\tiny
				$2 \displaystyle \sum_{i=6}^{n-2} \sum_{l_2=0}^{n-i-2} \sum_{l_1=2+l_2 }^{4+l_2} \sum_{k=l_1}^{l_2+4}  \sum_{m=k+l_1-l_2}^{n-j+k}  \frac{l_1 - l_2}{n-i} {n-i \choose l_1} {n-i \choose l_2}    \frac{l_1 - l_2}{n-j} {n-j \choose m-k} {n-j \choose m-k-l_1+l_2} {}_{A_{l_2}}{{N}}_{A_{l_1}}^{A_k} $         
			\end{center}
			where 
			\[
			{}_{A_{l_2}}{{N}}_{A_{l_1}}^{A_k}  = 
			\begin{cases}
				1 & \text{if } l_1=2+l_2, ~  k = l_1= l_2+2 \\
				2 & \text{if } l_1=2+l_2,  k=l_1+1=l_2+3\\
				1 & \text{if } l_1=2+l_2,  k=l_1+2=l_2+4\\
				2 & \text{if } l_1=3+l_2,  k=l_1=l_2+3\\
				2 & \text{if } l_1=3+l_2,  k=l_1+1=l_2+4\\
				1 & \text{if } l_1=4+l_2,  k=l_1=l_2+4\\
				0 & \text{ otherwise} 
			\end{cases}
			\]
			Here note that $i \geq 6$ as $|\gamma|=4$ and $|\beta|\geq 2$. 
			
			\begin{figure}[htbp]
				\centering
				\begin{minipage}{0.3\textwidth}
					\centering
					\includegraphics[scale=0.5]{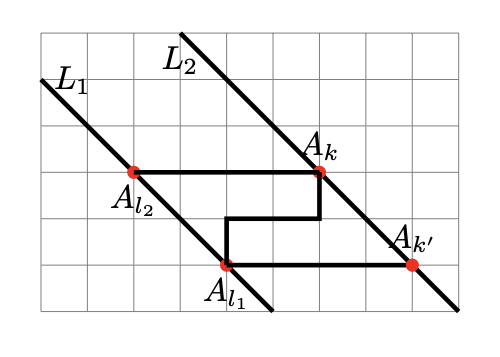}
					\caption{Illustration to the proof of Proposition \ref{Inticase} when $|\gamma|=4$, $k=l_2+4$,  $l_1-l_2=2$}
					\label{0_k=l_2+4_l_1=l_2+2}
				\end{minipage}
				\hspace{.2cm}
				\begin{minipage}{0.3\textwidth}
					\centering
					\includegraphics[scale=0.5]{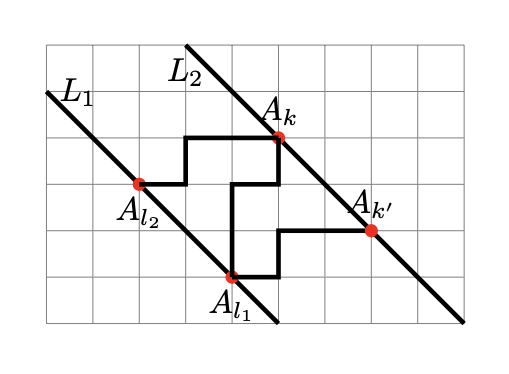}
					\caption{Illustration to the proof of Proposition \ref{Inticase} when $|\gamma|=4$, $k=l_2+3$,  $l_1-l_2=2$}
					\label{0_k=l_2+3_l_1=l_2+2}
				\end{minipage}
				\hspace{.2cm}
				\begin{minipage}{0.3\textwidth}
					\centering
					\includegraphics[scale=0.5]{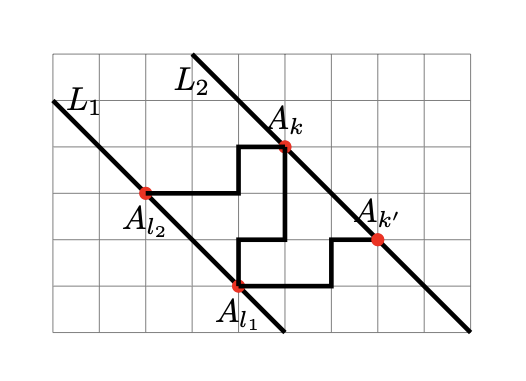}
					\caption{Illustration to the proof of Proposition \ref{Inticase} when $|\gamma|=4$, $k=l_2+3$,  $l_1-l_2=2$}
					\label{0_k=l_2+3_l_1=l_2+2_1}
				\end{minipage}
			\end{figure}
			\begin{figure}[htbp]
				\centering
				\begin{minipage}{0.3\textwidth}
					\centering
					\includegraphics[scale=0.5]{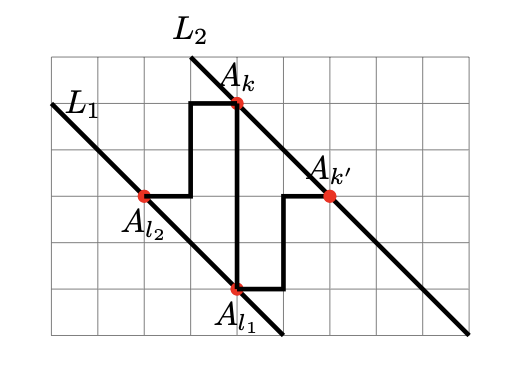}
					\caption{Illustration to the proof of Proposition \ref{Inticase} when $|\gamma|=4$, $k=l_2+2$, $l_1-l_2=2$}
					\label{0_k=l_2+2_l_1=l_2+2}
				\end{minipage}
				\hspace{.2cm}
				\begin{minipage}{0.3\textwidth}
					\centering
					\includegraphics[scale=0.5]{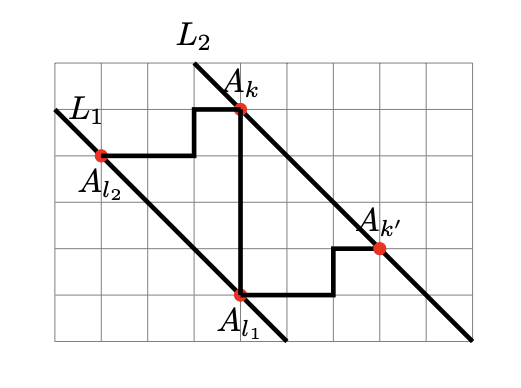}
					\caption{Illustration to the proof of Proposition \ref{Inticase} when $|\gamma|=4$, $k=l_2+3$, $l_1-l_2=3$}
					\label{0_k=l_2+3_l_1=l_2+3}
				\end{minipage}
				\hspace{.2cm}
				\begin{minipage}{0.3\textwidth}
					\centering
					\includegraphics[scale=0.5]{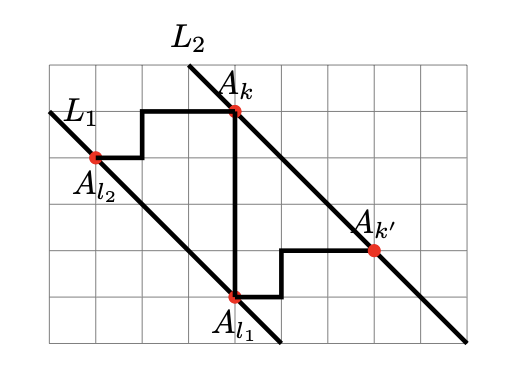}
					\caption{Illustration to the proof of Proposition \ref{Inticase} when $|\gamma|=4$, $k=l_2+3$, $l_1-l_2=3$}
					\label{0_k=l_2+3_l_1=l_2+3_1}
				\end{minipage}
			\end{figure}
			\begin{figure}[htbp]
				\centering
				\begin{minipage}{0.3\textwidth}
					\centering
					\includegraphics[scale=0.5]{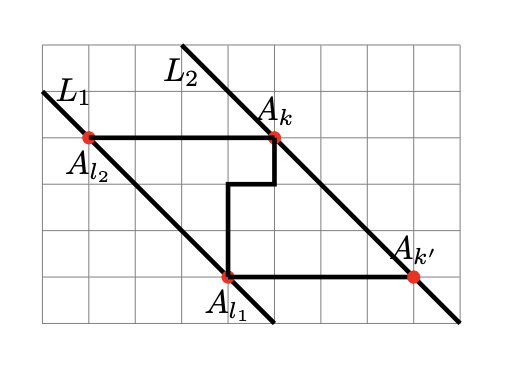}
					\caption{Illustration to the proof of Proposition \ref{Inticase} when $|\gamma|=4$, $k=l_2+4$,  $l_1-l_2=3$}
					\label{0_k=l_2+4_l_1=l_2+3}
				\end{minipage}
				\begin{minipage}{0.3\textwidth}
					\centering
					\includegraphics[scale=0.5]{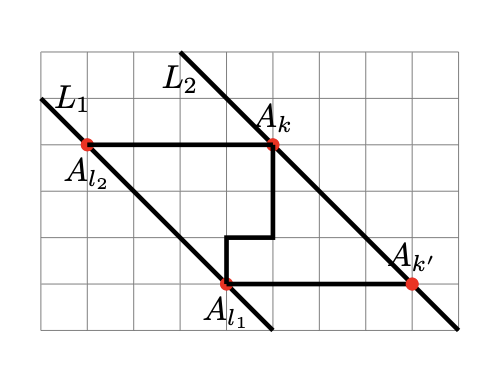}
					\caption{Illustration to the proof of Proposition \ref{Inticase} when $|\gamma|=4$, $k=l_2+4$,  $l_1-l_2=3$}
					\label{0_k=l_2+4_l_1=l_2+3_1}
				\end{minipage}
				\begin{minipage}{0.3\textwidth}
					\centering
					\includegraphics[scale=0.5]{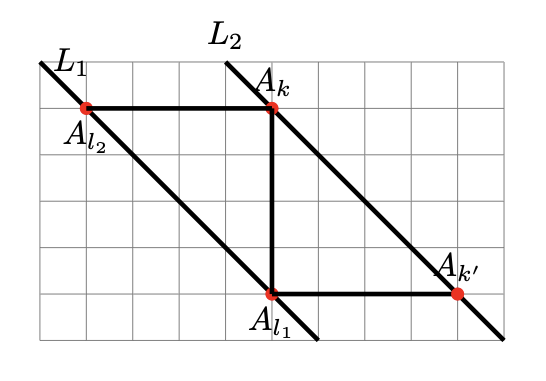}
					\caption{Illustration to the proof of Proposition \ref{Inticase} when $|\gamma|=4$, $k=l_2+4$, $l_1-l_2=4$}
					\label{0_k=l_2+4_l_1=l_2+4}
				\end{minipage}
			\end{figure}

		\end{itemize}
    \end{example}
	
	We have verified Example \ref{gamma example 45} through a computer program, and the results are given in Table \ref{examgamma45}.
	\begin{table}[h]
		\small
		\centering
		\begin{tabular}{|c|c|c|c|c|c|c|c|c|c|c|c|c|}
			 \hline
			 $n$ & 1 & 2 & 3 & 4 & 5 & 6 & 7 & 8 & 9 & 10 & 11 & 12 \\
			\hline
			$\mathcal{M}^{\textup{o}}(n)$ for $|\gamma|=3$  & 0 & 0 & 0 & 0 & 0 & 0 & 4 & 32 & 178 & 856 & 3820 & 16320   \\
			\hline
			$\mathcal{M}^{\textup{o}}(n)$ for $|\gamma|=4$  & 0 & 0 & 0 & 0 & 0 & 0 & 0 & 8 & 64 & 360 & 1760 & 8002   \\
			\hline
		\end{tabular}
		\caption{Values of $\mathcal{M}^{\textup{o}}(n)$, for $3 \le |\gamma| \le 4$  and $1 \leq n \leq 12$.}
		\label{examgamma45}
	\end{table}

	\section{Computing MAU pairs}

    Let $\overline{\mathcal{M}}(n)$  denote the number of mutually abelian-unbordered pairs of binary words $(u, v)$ where $|u|=|v|=n$. The definitions clearly show that $\overline{\mathcal{M}}(1)=4$. We now compute $\overline{\mathcal{M}}(n)$ for $n \geq 2$.
	For $n \geq 2$, let $\overline{\mathcal{M}}_{=}(n)$ denote the number of mutually abelian-unbordered pairs of binary words $(u, v)$ where $|u|=|v|=n$ and $ |u|_c = |v|_c$ for all $c \in \{a, b\}$. 
	For $n \geq 2$, let $\overline{\mathcal{M}}_{\neq}(n)$ denote the number of mutually abelian-unbordered pairs of binary words $(u, v)$ where $|u|=|v|=n$ and $ |u|_c \neq |v|_c$ for some $c \in \{a, b\}$. 
	Then for $n \geq 2$, $\overline{\mathcal{M}}(n) = \overline{\mathcal{M}}_{=}(n) + \overline{\mathcal{M}}_{\neq}(n)$. We now compute $\overline{\mathcal{M}}_{=}(n)$ and $\overline{\mathcal{M}}_{\neq}(n)$.

	\begin{proposition}\label{equalab}
		For $n \geq 2$, the value of $\overline{\mathcal{M}}_{=}(n)$ is 
		$$2 \sum_{i=1}^{n-1} \frac{1}{n-1} {n-1 \choose i} {n-1 \choose i-1}.$$
	\end{proposition}
	\begin{proof}
		Let $(u, v)$ be a mutually abelian-unbordered pair of binary words such that $|u| = |v| = n \geq 2$ and $|u|_c = |v|_c$ for all $c \in \{ a, b \}$. Let $|u|_a = |v|_a = r $ and $|u|_b = |v|_b =n-r$. Here $ 1 \leq r \leq n-1$, otherwise $(u, v)$ is not mutually abelian-unbordered.
		Then the lattice paths $p_O(u)$ and $p_O(v^R)$ intersect the line $X+Y=n$ at the same point, say $B$ where $B=(r,n-r)$ (see Fig. \ref{unbordered_abelian_u,v}).
		Now, from \cite{gessel1996counting}, the number of ordered pairs of non-intersecting lattice paths between $(0, 0)$ and $(r, n-r)$ is $ \frac{2}{n-1} {n-1 \choose r} {n-1 \choose r-1}$. This gives the number of such pairs $(u, v)$ where $|u|_a = |v|_a = r$.
		Therefore, $\overline{\mathcal{M}}_{=}(n)$  is  $2 \displaystyle \sum_{r=1}^{n-1} \frac{1}{n-1} {n-1 \choose r} {n-1 \choose r-1}$.
	\end{proof}

	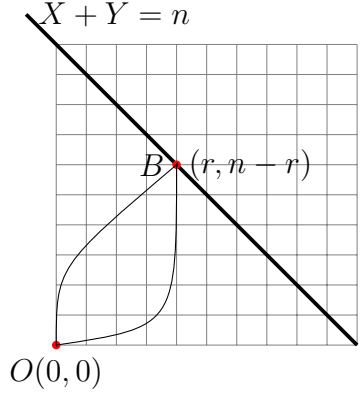
\begin{figure}[htbp]
		\centering
		\begin{tikzpicture}[scale=.8]
			\draw[step=0.5cm,gray,line width=.1mm] (0.5,0.5) grid (5.5,5.5);
			
			\draw[line width=.5mm]  (5.5,0.5) -- (0,6) node[right] {$X+Y=n$};
			
			\fill[red] (2.5,3.5) circle (2pt);
			\draw[line width=.5mm]  (2.5,3.5) -- (2.5,3.5) node[midway, right] {$(r, n-r)$};
			\draw[line width=.5mm]  (2.5,3.5) -- (2.5,3.5) node[left] {$B$};
			\fill[red] (0.5,0.5) circle (2pt);
			\draw[line width=.5mm]  (0.5,0.5) -- (0.5,0.5) node[midway, below] {$O(0, 0)$};

			\draw (0.5,0.5) .. controls (2.5,0.8) .. (2.5,3.5) node[midway, right] {};
			\draw (0.5,0.5) .. controls (.5,1.8) .. (2.5,3.5) node[midway, right] {};
		\end{tikzpicture}
		\caption{Illustration to the proof of Proposition \ref{equalab}}
		\label{unbordered_abelian_u,v}
	\end{figure}
	
	We have verified Proposition \ref{equalab} through a computer program and the results are given in Table \ref{tab:sample_table1213hgf}.
	\begin{table}[h]
		\centering
		\begin{tabular}{|c|c|c|c|c|c|c|c|c|c|c|c|c|}
			\hline
			$n$ & 2 & 3 & 4 & 5 & 6 & 7 & 8 & 9 & 10 & 11 & 12 \\
			\hline
			$\overline{\mathcal{M}}_{=}(n)$ & 2 & 4 & 10 & 28 & 84 & 264 & 858 & 2860 & 9724 &  33592 & 117572 \\
			\hline
		\end{tabular}
		\caption{Values of  $\overline{\mathcal{M}}_{=}(n)$ for $1 \leq n \leq 12$}
		\label{tab:sample_table1213hgf}
	\end{table}

	\begin{proposition}\label{poa1un}
		For $u, v \in \Sigma^n$ with $n \geq 2$, if $||u|_c - |v|_c| = 1 $ for some $ c \in \{a, b\}$, then $(u, v)$ can not be a mutually abelian-unordered pair of binary words.
	\end{proposition}
	\begin{proof}
		Let $||u|_c - |v|_c| = 1 $ for some $ c \in \{a, b\}$. Without loss of generality, let $c=a$. Then  $||u|_a - |v|_a| = 1 $. Since $|u|=|v|$,  $||u|_b - |v|_b| = 1 $. Thus, $||u|_c - |v|_c| = 1 $ for all $ c \in \{a, b\}$. Then we have the following cases:
		\begin{itemize}
			\item \textit{Case 1:} Consider $n=2$:
			
			Then $(u, v) \in \{ (ab, bb), (ab, aa), (ba, aa), (ba, bb), (aa, ab), (aa, ba), (bb, ab), (bb, ba)\}$. This implies $(u, v)$ can not be a mutually abelian-unordered pair of binary words.
			
			\item \textit{Case 2:} Consider $n \geq 3$ : 
			
			If possible, let there exists an MAU pair of words $(u, v)$ such that $||u|_a - |v|_a| = 1$ and $||u|_b - |v|_b| = 1$. Without loss of generality, let $|u|_a - |v|_a = 1$, and $|v|_b - |u|_b = 1$. Now, $(u, v)$ is a mutually abelian-unordered pair implies $\text{Pref}_1(u) \neq \text{Suff}_1(v)$ and $\text{Suff}_1(u) \neq \text{Pref}_1(v)$.
			Then we have the following cases:
			\begin{itemize}
				\item[(i)] If $u = a u' a$, $v = b v' b$ for some $u', v' \in \Sigma^+$, then as $|u|_a - |v|_a = 1$, $au' \sim_{\textup{abl}} v'b$, which is a contradiction as $(u, v)$ is a mutually abelian-unordered.
				
				\item[(ii)] If $u = a u' b$, $v = a v' b$ for some $u', v' \in \Sigma^+$, then as $|v|_b - |u|_b = 1$, $u'b \sim_{\textup{abl}} a v'$, which is a contradiction.
				
				\item[(iii)] If $u = b u' a$, $v = b v' a$ for some $u', v' \in \Sigma^+$, then as  $|u|_a - |v|_a = 1$, $b u' \sim_{\textup{abl}} v' a$, we have the contradiction. 
				
				\item[(iv)] Let $u = b u' b$, $v = a v' a$ for some $u', v' \in \Sigma^+$. Consider  $p_O(u)$ and $p_O(v^R)$ intersect the line $X+Y=n$ at the points $B$ and $B'$ respectively. Since $|u|_a - |v|_a = 1$, $X(B)-X(B')=1$. Then it can be observed that (see Figure \ref{unbornotexist}) the paths  $p_O(u)$ and $p_O(v^R)$ intersect after the origin $O$, which contradicts the fact that $(u, v)$ is an MAU pair.

\begin{figure}[h]
				\centering
				\begin{tikzpicture}
					\draw[step=0.5cm,gray,very thin] (0,0.5) grid (5,5);

					\draw[line width=.5mm]  (2,5) -- (2.5,4.5) node[below] {};
                    \draw[line width=.5mm]  (2.5,4.5) -- (4.5,2.5) node[below] {};
                    \draw[line width=.5mm]  (4.5,2.5) -- (5,2) node[below] {};
					\fill[red] (3,4) circle (2pt);
					\draw[line width=.5mm]  (3,4) -- (3,4) node[right] {$B'$};
                    \fill[red] (3.5,3.5) circle (2pt);
					\draw[line width=.5mm]  (3.5,3.5) -- (3.5,3.5) node[right] {$B$};
					\draw[line width=.5mm]  (2,5) -- (2,5) node[below,right] {$X+Y=n$};

					\fill[red] (0.5,1.5) circle (2pt);
					\draw[line width=.5mm]  (0.5,1.5) -- (0.5,1.5) node[below] {O};

                    \draw[line width=.5mm]  (0.5,1.5) -- (0.5,2) node[below] {};
                    \draw[line width=.5mm]  (3.5,3) -- (3.5,3.5) node[below] {};
                    \draw[dotted, line width=0.5mm] (0.5,2) .. controls (1,2.5) .. (3.5,3) node[midway, right] {};  
                    \draw[line width=.5mm]  (0.5,3) -- (0.5,3) node[below] {$p_O(u)$};

                    \draw[line width=.5mm]  (0.5,1.5) -- (1,1.5) node[below] {};
                    \draw[line width=.5mm]  (2.5,4) -- (3,4) node[below] {};
                    \draw[dotted,line width=0.8mm] (1,1.5) .. controls (2,2.5) .. (2.5,4) node[midway, right] {};  
                    \draw[line width=.5mm]  (2,2) -- (2,2) node[below] {$p_O(v^R)$};

				\end{tikzpicture}
				\caption{Illustration of proof of $(iv)$ point of Case $2$ of Proposition \ref{poa1un}. Here $p_O(u)$ is thin dotted line and $p_O(v^R)$ is the thick dotted line.}
				\label{unbornotexist}
			\end{figure}

			\end{itemize}
			
		\end{itemize}
		Thus, our assumption was wrong, i.e., there does not exist any mutually abelian-unordered pair of words $(u, v)$ such that $||u|_a - |v|_a| = 1$ and $||u|_b - |v|_b| = 1$.
	\end{proof}


		\begin{figure}[h!]
			\centering
			\begin{tikzpicture}[scale=1]
				
				\draw[dotted, ->, line width=0.15mm] (-2,0) -- (5,0) node[midway, above] {};
				\draw[line width=.1mm] (5,0) -- (5,0) node[midway,below] {$X$};
				\draw[dotted, ->, line width=0.15mm] (0,-0.6) -- (0,5) node[midway, above] {};
				\draw[line width=.1mm] (0,5) -- (0,5) node[midway, above] {$Y$};
				
				\draw[line width=.5mm]  (4,0) -- (0,4) node[right] {$X+Y=n$};
				\draw[line width=.5mm]  (0.5,-0.5) -- (-2,2) node[above] {$X=-Y$};
				
				\fill[red] (0,0) circle (2pt);
				\draw[line width=.5mm]  (0,0) -- (0,0) node[midway, below] {$O(0, 0)$};
				\fill[red] (-1.5,1.5) circle (2pt);
				\draw[line width=.5mm] (-1.5,1.5) -- (-1.5,1.5) node[midway, below] {$A$};
				\fill[red] (1,3) circle (2pt);
				\draw[line width=.5mm] (1,3) -- (1,3) node[midway, above] {$B_{r_2}$};     

				\draw (0,0) .. controls (0.5,1) .. (1,3) node[midway,above] {$p_O({v^R})$};
				\draw (-1.5,1.5) .. controls (-0.3,2.5) .. (1,3) node[midway, above] {$p_A(u)$};
				\draw (0,0) .. controls (2,1) .. (3,1) node[midway,below] {$p_O(u)$};
				\fill[red] (3,1) circle (2pt);
				\draw[line width=.5mm] (3,1) -- (3,1) node[midway, above] {$B_{r_1}$};

			\end{tikzpicture}
			\caption{Illustration to the proof of Proposition \ref{poa2un}}
			\label{fig12iolp}
		\end{figure}
		
		\begin{proposition}\label{poa2un}
			The number of mutually abelian-unbordered pair of words $(u, v)$ where $|u|=|v|=n$ and $||u|_c - |v|_c| \geq 2$ for all $c \in \{ a, b \}$ is 
			$$ 2  \sum_{r_2=0}^{n-2} \sum_{r_1=r_2+2}^{n} {}_{A}{{N}}_{{O}}^{B_{r_2}}$$
			where $ 2 \leq r_1-r_2 \leq n$, $A=(r_2-r_1, r_1-r_2)$, $O=(0, 0)$, and $B_{r_2}=(r_2, n-r_2)$.
		\end{proposition}
		\begin{proof}
			Let $(u, v)$ be a pair of binary words such that $|u|=|v|=n$ and $||u|_c - |v|_c| \geq 2$ for all $c \in \{ a, b \}$.
			Let $|u|_a - |v|_a \geq 2$.
			Consider $|u|_a = r_1$, $|v|_a=r_2$ for some $r_1, r_2 \geq 0$. Then  $ r_1 -r_2 \geq 2$.  
			Since $|u|=|v|$, $|v|_b - |u|_b= |u|_a - |v|_a= r_1 -r_2 \geq 2$. Let $O$ be the origin and $A=(r_2-r_1, r_1-r_2)$ be a point on line $X+Y=0$. Now we draw the lattice path $p_A(u)$. Let $p_A(u)$ intersect $X+Y=n$ at a point $B_{r_2}$. Then as  $|u|_a=r_1$, $X(B_{r_2})= r_1-(r_1 - r_2)= r_2$, i.e., $B_{r_2}=(r_2, n-r_2)$. Now we draw $p_O(u)$. Let $p_O(u)$ intersect the line $X+Y=n$ at a point $B_{r_1}$. Then as $|u|_a=r_1$, $B_{r_1}=(r_1, n-r_1)$. Now we draw $p_O({v^R})$. Since $|v|=n$ and $|v|_a=r_2$,  the path $p_O({v^R})$ intersect $X+Y=n$ at $B_{r_2}$.
			
			We now show that $(u, v)$ is a mutually abelian-unbordered pair of binary words iff  $p_A(u) \cap p_O({v^R}) = \{B_{r_2} \}$ and $p_O({v^R}) \cap  p_O(u)= \{O\}$ (see Fig. \ref{fig12iolp}).
			
			From Remark \ref{onb} (\ref{onb2}), $(u, v)$ has no external abelian-border iff $p_O({v^R}) \cap p_O(u)= \{O \}$.
			We now prove that  $(u, v)$ has no internal abelian-border iff  $p_A(u) \cap p_O({v^R}) = \{B_{r_2} \}$.
			We actually prove its contrapositive statement, i.e.,  $(u, v)$ has an internal abelian-border iff $p_A(u)$ and $p_O(v^R)$ intersect at a point other than $B_{r_2}$.

			Let $(u, v)$ have an internal abelian-border. Then $u=u' u''$ and $v^R=v'v''$ for some $u', u'', v', v'' \in \Sigma^+$ such that $u'' \sim_{\textup{abl}} v''$. Let $p_O(v')$ intersect the line $X+Y=n-|v''|$ at a point $C$. Then $C=(|v'|_a, |v'|_b)$.
			Now, $|u|_a - |v|_a = |v|_b - |u|_b= r_1 -r_2$ implies $|u'|_a - |v'|_a = r_1-r_2$ and $|v'|_b - |u'|_b= r_1 -r_2$. 
			Then $|u'|_a = |v'|_a + r_1-r_2$, $|v'|_b = |u'|_b +  r_1 -r_2$.
			Since $A=(r_2-r_1, r_1-r_2)$, $p_A({u'})$ intersect the line $X+Y=n-|v''|$ at the point $C$. 
			Thus, $C \in p_A(u) \cap p_O(v^R)$. Since $u', u'', v', v'' \in \Sigma^+$, $p_A(u)$ and $p_O(v^R)$ intersect at a point other than $B_{r_2}$.

			Conversely, let $ C' \in p_A(u) \cap p_O(v^R)$ where $C' \neq B_{r_2}$. 
			Then as $B_{r_2} \in p_A(u) \cap p_O(v^R)$, 
			$u$ has a non-empty proper suffix $u_1$ and $v^R$ has a non-empty proper suffix $v_1$ such that $u_1 \sim_{\textup{abl}} v_1$. Now, $v_1 \in \text{Suff}(v^R)$ implies $v_1^R \in \text{Pref}(v)$. So, $(u_1, v_1^R)$ is an internal abelian-border of $(u, v)$. 
			
			Therefore, $(u, v)$ is a mutually abelian-unbordered pair of binary words iff then $p_A(u) \cap p_O(v^R) = \{B_{r_2} \}$ and $p_O(v^R) \cap p_O(u)= \{O \}$.
			Now, ${}_{A}{{N}}_{{O}}^{B_{r_2}}$ gives the number of all triplets $(p_A(u), p_O(v^R), p_O(u))$ such that $p_A(u) \cap p_O(v^R) = \{B_{r_2} \}$ and $p_O(v^R) \cap p_O(u)= \{O \}$. Therefore 
			the number of mutually abelian-unbordered pair of binary words $(u, v)$ where $|u|_a = r_1$, $|v|_a=r_2$ and  $ r_1 -r_2 \geq 2$ is  ${}_{A}{{N}}_{{O}}^{B_{r_2}}$. 
			Similar things happen when $|v|_a - |u|_a \geq 2$.
			
			Therefore, 
			the number of mutually abelian-unbordered pair of words $(u, v)$ where $|u|=|v|=n$ and $||u|_c - |v|_c| \geq 2$ for all $c \in \{ a, b \}$ is 
			$2 \displaystyle \sum_{r_2=0}^{n-2} \sum_{r_1=r_2+2}^{n} {}_{A}{{N}}_{{O}}^{B_{r_2}}$
			where $ 2 \leq r_1-r_2 \leq n$, $A=(r_2-r_1, r_1-r_2)$, $O=(0, 0)$, and $B_{r_2}=(r_2, n-r_2)$.

		\end{proof}

		\begin{theorem}\label{unequunbo}
			From propositions \ref{poa1un} and \ref{poa2un}, we have the following:
			$$\overline{\mathcal{M}}_{\neq}(n) =  2 \sum_{r_2=0}^{n-2} \sum_{r_1=r_2+2}^{n} {}_{A}{{N}}_{{O}}^{B_{r_2}}$$
			where $ 2 \leq r_1-r_2 \leq n$, $A=(r_2-r_1, r_1-r_2)$, $O=(0, 0)$, and $B_{r_2}=(r_2, n-r_2)$.
		\end{theorem}

    \section{Conclusion}\label{concl}
	
	
	
	For $u, v \in \Sigma^n$, it is clear that if $(u,v)$ is an external abelian-bordered pair, then $(v, u)$ is internal abelian-bordered pair, and vice versa. Since these are the pairs which are neither MAB pairs nor MAU pairs, the number of such pairs can be computed as consequence of this work as stated below:
	\begin{proposition}
		The number of external/internal abelian-bordered pairs $(u, v)$ with $|u| = |v| = n$ is 
		$\frac{2^{2n} - \mathcal{M}(n) -  \overline{\mathcal{M}}(n)}{2}$.
	\end{proposition}
	
	Further, the following questions arise naturally in this context:
	\begin{enumerate}
		\item Do the limits  $\displaystyle \lim _{n \rightarrow \infty} \frac{\mathcal{M}(n)}{2^{2n}}$ and $\displaystyle \lim _{n \rightarrow \infty} \frac{\overline{\mathcal{M}}(n)}{2^{2n}}$ exist?

		\item What is the number of mutually abelian-bordered pairs of binary words $(u, v)$ when $|u| \neq |v|$?
		
		\item What is the number of mutually abelian-bordered pairs of non-binary words? 
		
	\end{enumerate}

\bibliographystyle{abbrv}
\bibliography{ref.bib}


\end{document}